\documentclass[brochure,12pt]{bourbaki}
\usepackage[matrix,arrow]{xy}
\usepackage{amssymb,amsfonts,amsmath,footnote}
\usepackage{enumerate}
\usepackage[francais]{babel}
\usepackage{comment}
\usepackage{footmisc}  
\usepackage[latin1]{inputenc}
\usepackage[T1]{fontenc}
\usepackage{microtype}
\DeclareFontFamily{OT1}{pzc}{}
\DeclareFontShape{OT1}{pzc}{m}{it}{<-> s * [1.10] pzcmi7t}{}
\DeclareMathAlphabet{\mathpzc}{OT1}{pzc}{m}{it}

\makeatletter
\def\myrightarrow{{\setbox\z@\hbox{$\rightarrow$}\dimen0\ht\z@\multiply\dimen0 6\divide\dimen0 10\ht\z@\dimen0\box\z@}}
\def\myrightarrowfill@{\arrowfill@\relbar\relbar\myrightarrow}
\newcommand{\myxrightarrow}[2][]{\ext@arrow 0359\myrightarrowfill@{#1}{#2}}
\makeatother

\newcommand{\isoto}{\myxrightarrow{\,\sim\,}}
\DeclareMathOperator{\bA}{\mathbb{A}}
\DeclareMathOperator{\bP}{\mathbb{P}}
\DeclareMathOperator{\C}{\mathbb{C}}
\DeclareMathOperator{\Q}{\mathbb{Q}}
\DeclareMathOperator{\Z}{\mathbb{Z}}
\DeclareMathOperator{\sA}{\mathcal{A}}
\DeclareMathOperator{\sC}{\mathcal{C}}
\DeclareMathOperator{\sE}{\mathcal{E}}
\DeclareMathOperator{\sF}{\mathcal{F}}
\DeclareMathOperator{\sG}{\mathcal{G}}
\DeclareMathOperator{\sK}{\mathcal{K}}
\DeclareMathOperator{\sL}{\mathcal{L}}
\DeclareMathOperator{\sM}{\mathcal{M}}
\DeclareMathOperator{\sN}{\mathcal{N}}
\DeclareMathOperator{\sO}{\mathcal{O}}
\DeclareMathOperator{\sQ}{\mathcal{Q}}
\DeclareMathOperator{\sU}{\mathcal{U}}
\DeclareMathOperator{\sX}{\mathcal{X}}
\DeclareMathOperator{\sY}{\mathcal{Y}}

\DeclareMathOperator{\Spec}{Spec}
\DeclareMathOperator{\Proj}{Proj}

\DeclareMathOperator{\rg}{rg}
\DeclareMathOperator{\Isom}{Isom}
\DeclareMathOperator{\Ker}{Ker}
\DeclareMathOperator{\Coker}{Coker}

\DeclareMathOperator{\Hull}{Hull}
\DeclareMathOperator{\Supp}{Supp}

\DeclareMathOperator{\GL}{GL}
\DeclareMathOperator{\PGL}{PGL}

\DeclareMathOperator{\Diff}{Diff}
\DeclareMathOperator{\QHusk}{QHusk}
\DeclareMathOperator{\Quot}{Quot}

\DeclareMathOperator{\red}{red}

\DeclareMathOperator{\Pic}{Pic}
\DeclareMathOperator{\prof}{prof}

\date{Janvier 2019}
\bbkannee{71\textsuperscript{e} ann\'ee, 2018--2019}  
\bbknumero{1155}                                      
\title{R\'eduction stable en dimension sup\'erieure}
\subtitle{d'apr\`es Koll\'ar, Hacon--Xu...}
\author{Olivier Benoist}
\address{CNRS, DMA\\
\'Ecole normale sup\'erieure\\
45 rue d'Ulm\\
75230 Paris Cedex 05}
\email{olivier.benoist@ens.fr}

\begin{document}
\maketitle

\section*{Introduction}

  L'espace de modules $M_g$ des courbes lisses de genre $g\geqslant 2$ construit par Mumford~\cite{GIT} est une vari\'et\'e alg\'ebrique  dont les points complexes sont naturellement en bijection avec les classes d'isomorphisme de courbes projectives lisses complexes de genre~$g$ (nous renvoyons \`a~\cite{AJP} et \`a~\cite{KoMum} pour un aper\c{c}u de l'histoire de ce sujet).

 Que ce soit pour \'etudier les d\'eg\'en\'erescences de familles de courbes lisses ou la g\'eom\'etrie de la vari\'et\'e $M_g$ elle-m\^eme, il est utile de disposer d'une compactification projective de $M_g$ qui soit \textit{modulaire}, c'est-\`a-dire qui param\`etre encore des courbes alg\'ebriques, \'eventuellement singuli\` eres. Une telle compactification a \'et\'e construite par  Deligne et Mumford~\cite{DM} : c'est l'espace de modules  des courbes stables $\overline{M}_g$.

  La recherche d'espaces de modules analogues param\'etrant des vari\'et\'es de dimension sup\'erieure a suscit\'e de nombreux travaux.
Pour obtenir une th\'eorie similaire, on se restreint aux vari\'et\'es dont le fibr\'e canonique est ample \footnote{Par le biais de leurs mod\`eles canoniques, cela prend en compte toutes les vari\'et\'es de type g\'en\'eral. Des compactifications modulaires ont aussi \'et\'e construites, par d'autres m\'ethodes que celles expliqu\'ees ici, pour d'autres espaces de modules : vari\'et\'es ab\'eliennes~\cite{Aleva}, certaines vari\'et\'es de Fano~\cite{LWX}.}. 
 Le cas des surfaces a alors \'et\'e r\'esolu par Koll\'ar, Shepherd-Barron et Alexeev ~\cite{KSB, Kocomplete, Ale}, et Viehweg~\cite{Viehweg} a trait\'e le cas des vari\'et\'es lisses en dimension arbitraire.

 Le cas g\'en\'eral a fait l'objet d'avanc\'ees r\'ecentes,
 d\'ecrites dans ce rapport. Ces progr\`es sont dus au d\'eveloppement du programme des mod\`eles minimaux  par Birkar, Cascini, Hacon, McKernan et Xu~\cite{BCHM, HX, HMX}, \`a de nombreux travaux de Koll\'ar~\cite{Kocomplete,KoHH,Kosing,Kobook}, ainsi qu'\`a Fujino, Kov\'acs et Patakfalvi~\cite{Fujino, KoPa}. 

Nous expliquons tout d'abord une motivation pour ces travaux : obtenir des th\'eor\`emes de r\'eduction stable en dimension sup\'erieure (th\'eor\`emes~\ref{redstable} et~\ref{redstabledimsup}). Nous d\'efinissons ensuite les vari\'et\'es stables qui jouent dans ce cadre le r\^ole des courbes stables de Deligne et Mumford, et \'enon\c{c}ons le th\'eor\`eme d'existence des espaces de modules de vari\'et\'es stables (th\'eor\`eme~\ref{thedm}). Dans les troisi\`eme et quatri\`eme sections, nous esquissons enfin la preuve du th\'eor\`eme de r\'eduction stable et la construction de ces espaces de modules.

\medskip

 {\it Conventions.} Tous les sch\'emas sont des $\Q$-sch\'emas noeth\'eriens. Une vari\'et\'e est un sch\'ema s\'epar\'e de type fini sur un corps $k$ de caract\'eristique nulle, par exemple le corps $\C$ des nombres complexes.

\section{R\'eduction semi-stable et r\'eduction stable}

  Fixons dans cette section un morphisme propre et surjectif $f:\mathcal{X}\to B$ entre vari\'et\'es r\'eduites.
 Supposons $B$ int\`egre, et notons $\eta$ le point g\'en\'erique de $B$ et $\mathcal{X}_{\eta}$ la fibre g\'en\'erique de $f$.
  On voit $f$ comme une famille de vari\'et\'es alg\'ebriques param\'etr\'ee par les points de $B$.
Cette famille peut avoir de mauvaises propri\'et\'es : les fibres de $f$ peuvent ne pas toutes avoir la m\^eme dimension, \^etre tr\`es singuli\`eres... On est ainsi amen\'e \`a rechercher des mod\`eles birationnels $f':\mathcal{X}'\to B'$ de $f$ dont la g\'eom\'etrie et les singularit\'es sont contr\^ol\'ees.
Plus pr\'ecis\'ement, on recherche un diagramme commutatif:
\begin{equation}
\label{diagred}
\begin{aligned}
\xymatrix{
 \mathcal{X}'\ar_{f'}[dr]\ar@{-->}^{\phi}[r]
&  \mathcal{X}_{B'}\ar[r]\ar[d] & \mathcal{X} \ar^f[d] \\
&B'\ar^{\pi}[r]&B .
}
\end{aligned}
\end{equation}
dans lequel $B'$ est une vari\'et\'e int\`egre de point g\'en\'erique $\eta'$, le morphisme
$\pi:B'\to B$ est propre, g\'en\'eriquement fini et surjectif, le carr\'e est cart\'esien, $\phi_{\eta'}$ est birationnelle et $f'$ est propre.
Quelles propri\'et\'es peut-on alors imposer au morphisme $f'$ ?

\subsection{R\'eduction semi-stable}

Une premi\`ere r\'eponse est apport\'ee par le th\'eor\`eme de r\'eduction semi-stable de Kempf, Knudsen, Mumford et Saint-Donat~\cite[p.~53]{KKMS}.

\begin{theo}
\label{KKMS}
Si $\dim(B)=1$, on peut choisir $f':\mathcal{X}'\to B'$ comme dans~\eqref{diagred} de sorte que $B'$ et $\mathcal{X}'$ soient lisses et les fibres de $f'$ soient des diviseurs r\'eduits \`a croisements normaux stricts dans $\mathcal{X}'$.
\end{theo}

 L'assertion que les fibres sont r\'eduites (c'est-\`a-dire sans multiplicit\'es) est ici essentielle.
Quand la base a dimension arbitraire, on dispose encore d'un th\'eor\`eme de r\'eduction semi-stable, d\'emontr\'e dans une variante faible par Abramovich et Karu~\cite{AbraKaru} et en toute g\'en\'eralit\'e par Adiprasito, Liu et Temkin~\cite{ALT}.

\begin{theo}
On peut choisir $f':\mathcal{X}'\to B'$ comme dans~\eqref{diagred} de sorte que $B'$ et $\sX'$ soient lisses, et $f'$ soit plat \`a fibres r\'eduites.
\end{theo}
Les \'enonc\'es de~\cite{AbraKaru, ALT} sont plus pr\'ecis: on peut garantir que $f'$ soit munie d'une structure toro\"idale.
On en d\'eduit par exemple que les fibres de $f'$ sont Gorenstein~\cite[Proposition 6.5]{AbraKaru}.

Les th\'eor\`emes de r\'eduction semi-stable ci-dessus ont l'avantage de donner lieu \`a des familles $f':\mathcal{X}'\to B'$ dont l'espace total $\mathcal{X}'$ est lisse. Ils ont cependant plusieurs inconv\'enients. Ils sont fortement non uniques. Par exemple, dans le cadre du th\'eor\`eme~\ref{KKMS}, on peut sans dommage \'eclater un point de $\mathcal{X}'$ en lequel $f'$ est lisse. De cette mani\`ere, m\^eme si le morphisme $f$ est lisse (si $\mathcal{X}_{\eta}$ a \textit{bonne r\'eduction}), il se peut que $f'$ ne le soit pas. Ainsi, si les singularit\'es des fibres sont tr\`es contr\^ol\'ees, leur g\'eom\'etrie ne l'est pas du tout. Les th\'eor\`emes de r\'eduction stable apportent une solution \`a ce probl\`eme.

\subsection{R\'eduction stable pour les familles de courbes}

Le premier tel \'enonc\'e, pour les familles \`a un param\`etre de courbes, est d\^u \`a Deligne et Mumford~\cite{DM} (d'autres preuves ont \'et\'e donn\'ees, par exemple dans~\cite{AW,Temkin}).

\begin{defi}
\label{defcourbestable}
Une \textbf{courbe stable} est une vari\'et\'e projective connexe $C$ de dimension $1$ dont les singularit\'es sont au plus nodales et dont le faisceau dualisant $\omega_C$ est ample. Le genre de $C$ est l'entier $g(C)=h^0(C,\omega_C)$.
\end{defi}

\begin{theo}
\label{redstcourbes}
Si $\dim(B)=1$ et si $\mathcal{X}_{\eta}$ est une courbe stable, on peut choisir $f':\mathcal{X}'\to B'$ comme dans~\eqref{diagred} de sorte que $f'$ soit plat \`a fibres des courbes stables, et $\phi_{\eta'}$ soit un isomorphisme.

De plus, si $B'$ est fix\'ee, un tel $f':\sX'\to B'$ est unique.
\end{theo}

Le th\'eor\`eme~\ref{redstcourbes} s'applique en particulier quand $\sX_{\eta}$ est une courbe lisse de genre $\geqslant 2$. \`A la diff\'erence du th\'eor\`eme~\ref{KKMS}, il
ne restreint pas les singularit\'es de $\sX'$.
La g\'eom\'etrie des fibres de $f'$ est en revanche tr\`es contrainte.

Les \'enonc\'es d'unicit\'e et d'existence dans le th\'eor\`eme~\ref{redstcourbes} refl\`etent la s\'eparation et la propret\'e de l'espace de modules des courbes stables $\overline{M}_g$ (et m\^eme, plus pr\'ecis\'ement,
du champ de modules $\overline{\sM}_g$ des courbes stables). La propret\'e de $\overline{\sM}_g$
 implique \`a son tour un th\'eor\`eme de r\'eduction stable sur des bases de dimension arbitraire.

\begin{theo}
\label{redstdJ}
Si $\mathcal{X}_{\eta}$ est une courbe stable, on peut choisir $f':\mathcal{X}'\to B'$ comme dans \eqref{diagred} de sorte que $f'$ soit plat \`a fibres des courbes stables, et $\phi_{\eta'}$ soit un isomorphisme.
\end{theo}

\begin{proof}[Preuve]
  Soit $g$ le genre de $\sX_{\eta}$.  Il n'existe pas de famille universelle
  de courbes stables sur l'espace de modules $\overline{M}_g$. Il r\'esulte en
  revanche du lemme de Chow pour les champs de Deligne-Mumford
 ~\cite[Th\'eor\`eme 16.6]{LMB}, appliqu\'e au champ de modules
  $\overline{\mathcal{M}}_g$ des courbes stables, qu'il existe une famille
  plate $p:\sC\to Z$ de courbes stables de genre~$g$ telle que le morphisme
  induit $Z\to\overline{\sM}_g$ soit fini
  et surjectif.  Remarquons que $Z$ est propre par propret\'e de
  $\overline{\sM}_g$.  La courbe stable $\sX_{\eta}$ induit un morphisme
  $\eta\to \overline{\sM}_g$. Notons $\eta'$~une composante irr\'eductible du
  produit fibr\'e $\eta\times_{\overline{\sM}_g} Z$.  Soient $\widetilde{B}$
  la normalisation de $B$ dans $\eta'$ et $B'\to \widetilde{B}$ une
  modification r\'esolvant les ind\'etermin\'ees de l'application rationnelle
  naturelle $\widetilde{B}\dashrightarrow Z$. Le morphisme $f':\sX' \to B'$
  construit en changeant de base $p:\sC\to Z$ par le morphisme $B'\to Z$ a les
  propri\'et\'es requises.
\end{proof}

\medskip

Le th\'eor\`eme~\ref{redstdJ}, appliqu\'e \`a une famille de courbes balayant une vari\'et\'e arbitraire,
est un outil crucial dans la preuve du th\'eor\`eme d'alt\'eration des singularit\'es de de Jong~\cite{dJ} (voir plus pr\'ecis\'ement~\cite[\S 2.24, \S 5.13]{dJ} ou~\cite[\S 3.2.3]{Berthelot}).

\subsection{R\'eduction stable en dimension sup\'erieure}

Nous d\'efinirons plus loin une notion de vari\'et\'e stable (d\'efinition~\ref{defstable}) et de famille de vari\'et\'es stables ou famille stable (d\'efinition~\ref{deffam})
en dimension sup\'erieure, permettant de g\'en\'eraliser les th\'eor\`emes~\ref{redstcourbes} et~\ref{redstdJ}.

Pour l'instant, disons seulement qu'une vari\'et\'e propre et lisse est stable si et seulement si son fibr\'e canonique est ample. C'est une condition bien plus restrictive pour les vari\'et\'es de dimension $\geqslant 2$ que pour les courbes. Par exemple, les th\'eor\`emes ci-dessous ne s'appliquent pas aux familles de 
vari\'et\'es de Fano, de vari\'et\'es ab\'eliennes ou de surfaces~$K3$.

\begin{theo}
\label{redstable}
Si $\dim(B)=1$ et si $\mathcal{X}_{\eta}$ est une vari\'et\'e stable, il existe $f':\mathcal{X}'\to B'$ comme dans \eqref{diagred} tel que $f'$ soit une famille stable, et $\phi_{\eta'}$ soit un isomorphisme.

De plus, si $B'$ est fix\'ee, un tel $f':\sX'\to B'$ est unique.
\end{theo}

Ce th\'eor\`eme est d\^u \`a Hacon et Xu~\cite{HX} quand $\sX_{\eta}$ est normale et \`a Koll\'ar en g\'en\'eral~\cite{Kosing, Kobook} (voir \S\ref{secredst} pour plus de d\'etails).

Comme dans le cas des courbes, une cons\'equence g\'eom\'etrique du th\'eor\`eme~\ref{redstable} est la propret\'e des espaces de modules de vari\'et\'es stables (voir le th\'eor\`eme~\ref{thedm}). Une fois de tels espaces de modules construits (ce qui est significativement plus dur que pour les espaces de modules de courbes, comme on le verra au \S~\ref{consedm}), l'argument expliqu\'e dans la preuve du th\'eor\`eme
~\ref{redstdJ} permet d'obtenir un th\'eor\`eme de r\'eduction stable sur une base de dimension sup\'erieure. 

\begin{theo}
\label{redstabledimsup}
Si $\mathcal{X}_{\eta}$ est une vari\'et\'e stable, il existe $f':\mathcal{X}'\to B'$ comme dans \eqref{diagred} tel que $f'$ soit une famille stable, et $\phi_{\eta'}$ soit un isomorphisme.
\end{theo}

\section{Stabilit\'e}

Dans cette section, nous d\'efinissons et \'etudions les analogues en dimension sup\'erieure des courbes stables de Deligne et Mumford.

\subsection{Vari\'et\'es stables}

On peut penser aux courbes lisses de genre $g\geqslant 2$ qui ne sont pas hyperelliptiques comme plong\'ees, \`a l'aide de leur fibr\'e canonique, dans l'espace projectif $\bP_k^{g-1}$. Si l'on veut aussi prendre en compte les courbes hyperelliptiques, il faut plut\^ot consid\'erer leur plongement tricanonique dans $\bP_k^{5g-6}$. On voudra aussi penser aux vari\'et\'es stables de dimension sup\'erieure comme \'etant pluricanoniquement plong\'ees. Ce point de vue va impr\'egner toute la suite de ce texte. Il explique le r\^ole pr\'epond\'erant que vont jouer le faisceau canonique et ses puissances dans la d\'efinition des vari\'et\'es stables.

\subsubsection{Singularit\'es}

 Introduisons tout d'abord la classe des singularit\'es que ces vari\'et\'es stables pourront porter.

\begin{defi}
\label{defslc}
 Une vari\'et\'e $X$
est dite \`a \textbf{singularit\'es semi-log canoniques (slc)} si elle satisfait les conditions (i)-(v) suivantes.
\begin{enumerate}[(i)]
\item $X$ est r\'eduite et purement de dimension $d$,
\item $X$ est \`a croisements normaux doubles en codimension~1,
\item $X$ satisfait la condition $S_2$ de Serre,
\item il existe $m>0$ tel que $\omega_X^{[m]}$ soit inversible,
\item les discr\'epances des diviseurs au-dessus de $X$ sont $\geqslant -1$.
\end{enumerate}
 Si $X$ est de plus normale ou de mani\`ere \'equivalente par le crit\`ere de Serre, si $X$ v\'erifie :
\begin{enumerate}[$(i)'$]
\setcounter{enumi}{1}
\item X est r\'eguli\`ere en codimension $1$,
\end{enumerate}
on dit que $X$ est \`a \textbf{singularit\'es log canoniques (lc)}.
\end{defi}

Expliquons ces conditions. Que $X$ soit \`a croisements normaux doubles en codimension~1 signifie qu'il existe un ouvert $U\subset X$ dont le compl\'ementaire a codimension $\geqslant 2$, le long duquel
 $X$ est soit r\'eguli\`ere, soit localement isomorphe (pour la topologie \'etale ou, si $k=\C$, pour la topologie analytique) \`a la singularit\'e $\{xy=0\}\subset \mathbb{A}_k^{d+1}$. Qu'il soit n\'ecessaire d'autoriser de telles singularit\'es est d\'ej\`a apparent dans le cas des courbes stables.

La condition $S_2$ de Serre est la propri\'et\'e de Hartogs: elle stipule que les fonctions r\'eguli\`eres sur $X$ s'\'etendent au travers des ferm\'es $Z\subset X$ de codimension $\geqslant 2$. Plus pr\'ecis\'ement, si
$Z$ est un tel ferm\'e et si $j:X\setminus Z\hookrightarrow X$ est l'inclusion,
 le morphisme naturel $\sO_X\to j_*\sO_{X\setminus Z}$ est un isomorphisme. C'est un substitut de la normalit\'e de $X$.

Les vari\'et\'es stables doivent \^etre pens\'ees comme (pluri)canoniquement plong\'ees et il est donc important de contr\^oler les formes diff\'erentielles de degr\'e maximal sur $X$. C'est le r\^ole des conditions (iv) et (v). Notons $j:U\hookrightarrow X$ le plus gros ouvert le long duquel les singularit\'es de $X$ sont \`a croisements normaux doubles.
Comme les croisements normaux doubles sont des singularit\'es localement d'intersection compl\`ete, donc Gorenstein, le faisceau dualisant $\omega_U$ de $U$ est un faisceau inversible\footnote{Sur l'ouvert de lissit\'e de $X$, il s'agit du faisceau canonique des formes diff\'erentielles de degr\'e maximal. On peut d\'ecrire tr\`es concr\`etement $\omega_U$ en g\'en\'eral: une section locale est une $d$-forme diff\'erentielle sur la normalisation, \`a p\^oles au plus logarithmiques le long de l'image inverse du lieu double, et dont les r\'esidus le long des deux branches du lieu double sont oppos\'es~\cite[Proposition 5.8]{Kosing}.\label{fn:repeat}}. 
On d\'efinit le \textbf{faisceau canonique}\footnote{On prendra garde que, $X$ n'\'etant pas Cohen-Macaulay en g\'en\'eral, ce faisceau peut ne pas co\"incider avec le complexe dualisant de $X$ : il n'en est qu'un des faisceaux de cohomologie.} de $X$ par $\omega_X:=j_*\omega_U$ et on introduit, pour tout $n\in\Z$, ses puissances r\'eflexives $\omega_X^{[n]}:=j_*(\omega_U^{\otimes n})$ : les \textbf{faisceaux pluricanoniques} de $X$.
Qu'il existe un entier $m>0$ tel que $\omega_X^{[m]}$ soit un faisceau inversible, donc associ\'e \`a un fibr\'e en droites, est bien s\^ur une condition n\'ecessaire \`a toute tentative de voir $X$ comme plong\'ee \`a l'aide de formes pluricanoniques ! 

La condition (v) donne un contr\^ole birationnel sur les formes pluricanoniques sur $X$. Soit $\pi:Y\to X$ une modification normale\footnote{Une \textbf{modification} est un morphisme propre birationnel. On n'a pas vraiment besoin de supposer $Y$ normale : il suffit que $Y$ soit $S_2$ et r\'eguli\`ere aux points g\'en\'eriques des diviseurs exceptionnels de $\pi$.}  de $X$ (par exemple la normalisation de $X$ ou une r\'esolution des singularit\'es de $X$), et soient $(E_i)_{i\in I}$ les diviseurs exceptionnels de~$\pi$. Soit $m>0$ un entier tel que $\omega_X^{[m]}$ soit inversible. Au-dessus du lieu $Y\setminus \bigcup_i E_i$ o\`u $\pi$ est un isomorphisme, on dispose d'un isomorphisme \'evident $\rho:\omega^{[m]}_{Y\setminus \cup_i E_i}\isoto (\pi^*\omega_X^{[m]})|_{Y\setminus \cup_i E_i}$. Comme $\omega^{[m]}_{Y}$ et $\pi^*\omega_X^{[m]}$ sont inversibles au point g\'en\'erique de chacun des $E_i$, le morphisme~$\rho$ a des z\'eros ou des p\^oles d'une certaine multiplicit\'e le long de ces diviseurs, de sorte qu'il existe des $a_{E_i}(X)\in\frac{1}{m}\Z$ tels que $\rho$ se prolonge en un isomorphisme
\begin{equation}
\label{discrepeq}
\rho:\omega^{[m]}_{Y}\isoto \pi^*\omega_X^{[m]}(\sum_i m \cdot a_{E_i}(X)E_i).
\end{equation}
Les nombres rationnels $a_{E_i}(X)$ ont \'et\'e choisis pour ne pas d\'ependre du choix de l'entier~$m$ : ce sont les \textbf{discr\'epances} des diviseurs $E_i$. 

La condition (v) selon laquelle ces discr\'epances sont toujours $\geqslant -1$ signifie en substance
 que les formes canoniques sur $X$ s'\'etendent en des formes \`a p\^oles au plus logarithmiques sur les modifications de $X$. Il suffit de la v\'erifier pour les diviseurs apparaissant sur une r\'esolution arbitraire des singularit\'es de $X$ dont le diviseur exceptionnel est \`a croisements normaux stricts (combiner~\cite[Lemma 5.10 et Corollary 2.13]{Kosing}).
C'est la condition la plus subtile de la d\'efinition~\ref{defslc}.
La preuve transparente de l'unicit\'e dans le th\'eor\`eme de r\'eduction stable au \S~\ref{unicite} permet de se convaincre de sa pertinence.

\medskip

On d\'efinit d'autres classes de singularit\'es en conservant les conditions
(i)-(iv), mais en demandant \`a ce que les discr\'epances des diviseurs
au-dessus de $X$ soient $>-1$ (resp.{} $\geqslant 0$, resp.{} $>0$) : ce sont
les singularit\'es \textbf{kawamata log terminales} ou \textbf {klt} (resp.{}
\textbf{canoniques}, resp.{} \textbf{terminales}). Ces singularit\'es sont normales. Nous nous en servirons peu.

Ces d\'efinitions s'\'etendent sans difficult\'es \`a des sch\'emas plus g\'en\'eraux que des vari\'et\'es. Nous les utiliserons par exemple pour des sch\'emas de type fini sur le spectre d'un anneau de valuation discr\`ete au \S\ref{secredst} et au \S\ref{ouvertslc}.

\subsubsection{D\'efinition}

La notion de stabilit\'e combine les propri\'et\'es locales discut\'ees ci-dessus et une condition globale d'amplitude du faisceau canonique.
\begin{defi}
\label{defstable}
Une \textbf{vari\'et\'e stable} est une vari\'et\'e projective
$X$ \`a singularit\'es slc dont le faisceau canonique $\omega_X$ est ample.
\end{defi}

Le faisceau $\omega_X$ n'est pas inversible en g\'en\'eral. La d\'efinition~\ref{defstable} requiert seulement qu'il soit ample comme $\Q$-fibr\'e en droites, c'est-\`a-dire que $\omega_X^{[m]}$ soit ample pour un $m>0$ (de mani\`ere \'equivalente, pour tout $m>0$) tel que $\omega_X^{[m]}$ soit inversible.

Les courbes stables sont traditionnellement suppos\'ees connexes, comme dans la d\'efinition~\ref{defcourbestable}. Il est plus naturel de ne pas faire cette hypoth\`ese (voir par exemple le th\'eor\`eme~\ref{bij}).
La d\'efinition des vari\'et\'es stables dans le cas des surfaces avait \'et\'e d\'egag\'ee par Koll\'ar et Shepherd-Barron~\cite[\S 5.4]{KSB} et la d\'efinition en dimension arbitraire en est une extension imm\'ediate. En revanche, l'\'etude de ces vari\'et\'es est bien plus difficile en dimension $\geqslant 3$ qu'en dimension $2$.

\subsubsection{Cas des paires}
\label{paires}

Nous utiliserons la variante suivante des d\'efinitions~\ref{defslc} et~\ref{defstable}. Une \textbf{paire} $(X,\Delta)$ est constitu\'ee d'une vari\'et\'e $X$ et d'un $\Q$-diviseur de Weil
$\Delta=\sum c_i\Delta_i$, o\`u les $\Delta_i$ sont des sous-vari\'et\'es int\`egres de codimension $1$ de $X$ non incluses dans le lieu singulier de $X$ et o\`u $c_i\in\Q$ (dans la pratique, on aura m\^eme $c_i\in\Q\cap[0,1]$). On \'etend les d\'efinitions \`a ce cadre en rempla\c{c}ant partout $\omega_X$ par le faisceau canonique $\omega_X(\Delta)$ de la paire.

\begin{defi}
\label{slcpaires}
La paire $(X,\Delta)$ est \`a \textbf{singularit\'es slc} (resp. \textbf{lc}) si $c_i\in\Q\cap[0,1]$, si $X$ est r\'eduite, purement de dimension $d$, \`a croisements normaux doubles (resp. r\'eguli\`ere) en codimension $1$ et $S_2$, s'il existe un entier $m>0$ tel que $\omega_X^{[m]}(m\Delta)$ est inversible et si les discr\'epances $a_E(X,\Delta)$ des diviseurs $E$ au-dessus de $X$ sont $\geqslant -1$.

Elle est \textbf{stable} si elle est \`a singularit\'es slc si $X$ est projective et si $\omega_X(\Delta)$ est ample.
\end{defi}

Dans cette d\'efinition, les hypoth\`eses faites sur $X$ assurent l'existence d'un ouvert $j:U\hookrightarrow X$ dont le compl\'ementaire a codimension $\geqslant 2$ le long duquel $X$ est Gorenstein et les $\Delta_i$ sont Cartier. Si $n\in \Z$ est tel que les $nc_i$ sont des entiers, cela permet de d\'efinir le faisceau $n$-canonique $\omega_X^{[n]}(n\Delta):=j_*(\omega_U^{\otimes n}(n\Delta|_U))$ de $(X,\Delta)$.
Les discr\'epances $a_{E}(X,\Delta)$ sont calcul\'ees par rapport au faisceau canonique $\omega_X(\Delta)$ de la paire. Si $\pi:Y\to X$ est une modification normale de $X$ avec diviseurs exceptionnels $E_i$, si $(\pi^{-1})_*\Delta$
est la transform\'ee stricte de $\Delta$ dans $Y$, et si $m>0$ est tel que $\omega_X^{[m]}(m\Delta)$ est inversible, elles sont d\'efinies par l'isomorphisme naturel g\'en\'eralisant~\eqref{discrepeq} :
\begin{equation}
\label{logdiscrepeq}\omega^{[m]}_{Y}(m(\pi^{-1})_*\Delta)\isoto \pi^*\omega_X^{[m]}(m\Delta)\Big(\sum_i m \cdot a_{E_i}(X,\Delta)E_i\Big).
\end{equation}

Nous aurons \`a consid\'erer des paires pour plusieurs raisons ; la principale est la suivante. Soit $X$ une vari\'et\'e satisfaisant aux conditions (i)-(iv) de la d\'efinition~\ref{defslc}. Soit $m>0$ tel que $\omega^{[m]}_X$ soit inversible et soit $\nu:\widetilde{X}\to X$ la normalisation de $X$. 
Notons $\Gamma\subset\widetilde{X}$ le lieu exceptionnel de $\nu$, muni de sa structure r\'eduite. C'est un diviseur qui est l'adh\'erence de l'image inverse par $\nu$ du lieu o\`u $X$ est \`a croisements normaux doubles. On appelle $\Gamma$ le \textbf{conducteur} de $X$. L'isomorphisme \'evident $(\nu^*\omega_X^{[m]})|_{\widetilde{X}\setminus\Gamma}\isoto \omega_{\widetilde{X}}^{[m]}|_{\widetilde{X}\setminus\Gamma}$ se prolonge en un isomorphisme
\begin{equation}
\label{conducteur}
\nu^*\omega_X^{[m]}\isoto\omega_{\widetilde{X}}^{[m]}(m\Gamma),
\end{equation}
comme le montre un calcul local sur le lieu o\`u $X$ est \`a croisements
normaux doubles~\cite[(5.7.4)]{Kosing}. On d\'eduit imm\'ediatement de
l'isomorphisme~\eqref{conducteur} l'\'equivalence~\cite[Lemma 5.10]{Kosing} :
\begin{equation}
\label{slclc}
X\textrm{ est \`a singularit\'es slc }\iff (\widetilde{X},\Gamma)\textrm{ est \`a singularit\'es lc.}
\end{equation}
Ce proc\'ed\'e de normalisation permettra 
de ramener l'\'etude des vari\'et\'es \`a singularit\'es slc au cas normal.
Comprendre dans quelle mesure on peut reconstruire $X$ \`a partir de $(\widetilde{X},\Gamma)$ est une question difficile (voir le th\'eor\`eme~\ref{bij} pour un \'enonc\'e pr\'ecis).

\subsubsection{Exemples} 
\label{exslc}
Les seules singularit\'es slc de dimension $1$ sont les n\oe{}uds.

En dimension $2$, les singularit\'es slc ont \'et\'e classifi\'ees  par Kawamata~\cite[Theorem~2]{Kawsing} dans le cas normal et par Koll\'ar et Shepherd-Barron~\cite[Theorem 4.24]{KSB} en g\'en\'eral (voir aussi~\cite[\S 3.3]{Kosing} ou~\cite{Kobook}). Sans rappeler cette classification en d\'etail, donnons quelques exemples repr\'esentatifs. Les singularit\'es obtenues comme quotient de $\bA_k^2$ par un sous-groupe fini de $\GL_2(k)$ sont lc. Cela inclut toutes les singularit\'es Du Val (ou points doubles rationnels). D'autres singularit\'es lc de surface sont les singularit\'es elliptiques obtenues comme c\^ones sur une courbe elliptique.

Des exemples de surfaces slc non normales sont les points \`a croisement normaux triples $\{xyz=0\}\subset \bA_k^3$, le parapluie de Whitney ou pinch point $\{x^2=yz^2\}\subset \bA_k^3$, ou un c\^one sur une courbe elliptique nodale $\{y^2=x^3+x^2\}\subset \bA_k^3$.

On ne dispose pas de classification en dimension sup\'erieure. Les c\^ones
\begin{equation}
\label{cone}
C(X,L):=\Spec\bigoplus_{l\geqslant 0} H^0(X,L^{\otimes l})
\end{equation}
o\`u $X$ est une vari\'et\'e projective munie d'un fibr\'e ample $L$, fournissent une instructive source d'exemples. 
On calcule que $C(X,L)$ a des singularit\'es slc (resp. lc) si et seulement si $X$ a des singularit\'es slc (resp. lc) et s'il existe des entiers $m<0$ et $l\geqslant 0$ tels que $\omega_X^{[m]}\simeq L^{\otimes l}$~\cite[\S 3.1]{Kosing}. En particulier, le c\^one anticanonique sur une vari\'et\'e de Fano, ou un c\^one associ\'e \`a un fibr\'e ample arbitraire sur une vari\'et\'e de Calabi-Yau, ont des singularit\'es lc.

D'autres exemples \'el\'ementaires sont les singularit\'es quotient, c'est-\`a-dire les quotients de vari\'et\'es lisses\footnote{Il est faux en g\'en\'eral que le quotient d'une vari\'et\'e lc par un groupe fini est encore lc. Soit $\pi:S\to T$ une surface $K3$ obtenue comme rev\^etement double de 
$T=\bP^1\times\bP^1$ ramifi\'e au-dessus d'un diviseur lisse de bidegr\'e $(4,4)$, et notons $L:=\sO_T(1,2)$. Le morphisme de c\^ones $C(S,\pi^*L)\to C(T,L)$ est le quotient par une action de $\Z/2\Z$, mais $C(S,\pi^*L)$ est lc alors que $C(T,L)$ ne l'est pas.} par l'action d'un groupe fini~\cite[3.18]{Kosing}. Des exemples plus riches, \`a la topologie plus compliqu\'ee, ont \'et\'e construits par Koll\'ar~\cite{Konew}.

\subsection{Familles stables}

\subsubsection{D\'efinition} 
\label{pardeffamst}
Comme on le verra au \S\ref{exfam}, les familles plates \`a fibres slc (resp. stables) ne donnent pas lieu \`a une bonne notion de famille de vari\'et\'es slc (resp. stables). La raison pour cela est que, si l'on souhaite penser aux vari\'et\'es stables comme \'etant pluricanoniquement plong\'ees, il est important que les faisceaux (pluri)canoniques des fibres varient convenablement en famille ; c'est une condition que l'on doit imposer.

\begin{defi}
\label{deffam}
Une \textbf{famille localement stable} est un morphisme plat \`a fibres slc $f:\mathcal{X}\to B$  tel que pour tout $n\in\Z$, le faisceau $\omega_{\mathcal{X}/B}^{[n]}$ soit $f$-plat de formation commutant \`a tout changement de base.
C'est une \textbf{famille de vari\'et\'es stables} ou \textbf{famille stable} si $f$ est de plus propre \`a fibres stables.\end{defi}

 Dans cette d\'efinition, les faisceaux pluricanoniques relatifs $\omega_{\mathcal{X}/B}^{[n]}$ sont construits comme dans le cas absolu. Plus pr\'ecis\'ement, on note $j:\mathcal{U}\hookrightarrow\mathcal{X}$ le plus gros ouvert 
 le long duquel les singularit\'es des fibres g\'eom\'etriques de $f$ sont \`a croisements normaux doubles. Le morphisme $f|_{\mathcal{U}}$ est plat \`a fibres Gorenstein,
de sorte que le faisceau dualisant relatif $\omega_{\mathcal{U}/B}$ est inversible~\cite[Theorem 3.5.1]{Conrad}. On pose $\omega_{\mathcal{X}/B}:=j_*\omega_{\mathcal{U}/B}$ et $\omega_{\mathcal{X}/B}^{[n]}:=j_*(\omega_{\mathcal{U}/B}^{\otimes n})$.

La d\'efinition~\ref{deffam} requiert tout d'abord que les faisceaux pluricanoniques relatifs $\omega_{\mathcal{X}/B}^{[n]}$ soient plats sur $B$. Cette hypoth\`ese ne suffit pas \`a assurer 
que les fibres $\omega_{\mathcal{X}/B}^{[n]}|_{\sX_b}$ de ces faisceaux au-dessus d'un point $b\in B$ co\" incident avec les faisceaux pluricanoniques $\omega_{\mathcal{X}_b}^{[n]}$ de la fibre. C'est le r\^ole de la condition  de changement de base dans la d\'efinition~\ref{deffam} : elle revient \`a imposer que le morphisme naturel 
$\omega_{\mathcal{X}/B}^{[n]}|_{\sX_b}\to \omega_{\mathcal{X}_b}^{[n]}$ soit un isomorphisme, pour tout $b\in B$ et tout $n\in\Z$. Ceci implique\footnote{Pour le voir, on peut combiner~\cite[Th\'eor\`eme 5.10.5 et Proposition 6.3.1]{EGA42}.} en effet la propri\'et\'e, a priori plus forte, de commutation \`a tout changement de base : pour tout morphisme $g:B'\to B$, si l'on note $g_{\sX}:\sX'\to\sX$ le changement de base, le morphisme naturel
$g_{\sX}^*\omega_{\mathcal{X}/B}^{[n]}\to \omega_{\mathcal{X}'/B'}^{[n]}$
est un isomorphisme.

Il suit de la d\'efinition~\ref{deffam} que si $f$ est localement stable, il existe $m>0$ tel que le faisceau $\omega^{[m]}_{\sX/B}$ soit inversible (et $f$-ample si $f$ est stable). En effet, par r\'ecurrence noeth\'erienne sur la base $B$, on peut choisir $m$ de sorte que $\omega_{\mathcal{X}_b}^{[m]}$ soit inversible pour tout $b\in B$. Il r\'esulte de sa platitude et de sa commutation au changement de base que $\omega^{[m]}_{\sX/B}$ est inversible (et $f$-ample si les fibres sont stables). Une famille stable est donc bien canoniquement polaris\'ee, comme d\'esir\'e.

Les conditions de platitude et de commutation au changement de base pour $\omega_{\mathcal{X}/B}^{[n]}$ sont subtiles. Elles sont automatiques pour $n=1$ par~\cite[Theorem 7.9.3]{lcdB} et~\cite[Corollary 6.32]{Kosing}. Elles sont toujours v\'erifi\'ees si les fibres de $f$ sont \`a singularit\'es canoniques\footnote{Justifions-le. Par classification des singularit\'es canoniques de surfaces~\cite[Theorem 4.20]{KM}, les fibres de $f$ sont Gorenstein en codimension $2$. Par~\cite[Theorem 3.5.1]{Conrad}, il existe un ouvert $j:\mathcal{U}\hookrightarrow\mathcal{X}$ 
tel que $\sX\setminus\sU$ a codimension $\geqslant 3$ dans les fibres de $f$ et tel que $f|_{\sU}$ est Gorenstein, de sorte que $\omega_{\sU/B}$ est inversible. Comme les fibres de $f$ sont de plus $S_3$ par~\cite{Elkik} (voir le th\'eor\`eme~\ref{rationnelles}), on peut conclure \`a l'aide de~\cite[Theorem 12]{Koflat}.} (voir~\cite[Aside 30]{Kosurvey}).
Enfin, quand la base $B$ est r\'eduite, on dispose d'un crit\`ere num\'erique : il est \'equivalent de demander que le degr\'e de la polarisation canonique des fibres soit localement constant sur la base~\cite
{Kobook}.

  Dans la d\'efinition~\ref{deffam}, la condition de commutation de tous les $\omega_{\mathcal{X}/B}^{[n]}$ aux changements de base est connue sous le nom de \textit{condition de Koll\'ar}.
   Une variante, dite \textit{condition de Viehweg}~\cite[Assumptions 8.30]{Viehweg}, consiste \`a demander que $\omega_{\mathcal{X}/B}^{[m]}$ soit inversible (et par cons\'equent commute aux changements de base) seulement pour un $m>0$.
Elle permet \'egalement de construire des espaces de modules projectifs de vari\'et\'es stables ; ils diff\`erent par leur structure sch\'ematique de ceux
obtenus \`a l'aide de la condition de Koll\'ar (voir~\cite{Kodef}).

\subsubsection{Exemples} 
\label{exfam}
Illustrons, en suivant~\cite[7.A]{YPG} et~\cite[Example 26]{Kosurvey}, l'importance de la condition de changement de base dans la D\'efinition~\ref{deffam}.

 Consid\'erons d'une part le plongement de Veronese $\Sigma_1=\bP^2_{k}\hookrightarrow\bP^{5}_{k}$ et d'autre part le plongement $\Sigma_2=\bP^1_{k}\times\bP^1_{k}\hookrightarrow\bP^5_{k}$ induit par $\sO(1,2)$.  Ces deux surfaces projectives ont pour sections hyperplanes lisses des courbes rationnelles normales quartiques $\Gamma$. Pour $i\in\{1,2\}$, soit $f_i:\sX_i\to\bP^1_{k}$ un pinceau g\'en\'eral de sections hyperplanes du c\^one $C(\Sigma_i,\sO(1))$ sur $\Sigma_i$. Toutes les fibres de $f_i$ sont isomorphes \`a $\Sigma_i$, sauf la section hyperplane passant par le sommet du c\^one, qui est isomorphe au c\^one $C$ sur la courbe rationnelle normale quartique\footnote{Cette section hyperplane pourrait a priori avoir un point immerg\'e au sommet. On v\'erifie que ce n'est pas le cas en remarquant que $\Sigma_i$ et $C$ ont m\^eme polyn\^ome de Hilbert.}. On voit ainsi que les fibres de $f_i$ ont des singularit\'es lc (voir \S~\ref{exslc}).
 
  On remarque cependant que les nombres d'intersection  $\omega_{\Sigma_1}\cdot\omega_{\Sigma_1}=9$ et $\omega_{\Sigma_2}\cdot\omega_{\Sigma_2}=8$ des surfaces $\Sigma_1$ et $\Sigma_2$ diff\`erent. Comme $f_1$ et $f_2$ ont une fibre sp\'eciale isomorphe \`a $C$  en commun, cette remarque n'est pas compatible avec le fait que les faisceaux dualisants relatifs de ces familles forment un $\Q$-fibr\'e en droites. Cela s'explique par le fait que, si $f_1$~est bien localement stable (en particulier, le faisceau $\omega^{[2]}_{\sX_1/B}$ est inversible\footnote{Dans ces exemples, l'unique singularit\'e de l'espace total est celle d'un c\^one. On v\'erifie alors ces assertions \`a l'aide du calcul du groupe des classes d'un c\^one~\cite[Proposition 3.14 (4)]{Kosing}.\label{note1}}), la famille~$f_2$ ne l'est pas ($\omega^{[n]}_{\sX_2/B}$ n'est inversible pour aucun $n>0$\textsuperscript{(\ref{note1})}).
  
  En rempla\c{c}ant les fibres des $f_i$ par des rev\^etements ramifi\'es appropri\'es, on obtient des exemples analogues pour lesquels $f_1$ est stable (et pas seulement localement stable).
 
 \medskip
 
 On construit un exemple un peu diff\'erent en suivant~\cite[Example 5.12]{KSB}. Effectuons la m\^eme construction \`a l'aide des deux surfaces $\Sigma'_1=\bP^1_{k}\times\bP^1_{k}\hookrightarrow\bP^8_{k}$ et $\Sigma'_2=\bP_{\bP^1_{k}}(\sO\oplus\sO(1))\hookrightarrow\bP^8_{k}$, plong\'ees par leur fibr\'e anticanonique, dont les sections hyperplanes lisses sont des courbes elliptiques octiques. Prenant, pour $i\in\{1,2\}$, un pinceau de sections hyperplanes du c\^one $C(\Sigma'_i,\sO(1))$ sur $\Sigma'_i$, on peut obtenir deux familles $f'_i:\sX'_i\to\bP^1_{k}$ dont les fibres g\'en\'erales sont toutes isomorphes \`a $\Sigma'_i$, sauf une qui est un c\^one sur une courbe elliptique octique fix\'ee. \`A la diff\'erence de l'exemple pr\'ec\'edent, les deux familles sont localement stables : on v\'erifie m\^eme que $\omega_{\sX^i/B}$ est inversible\textsuperscript{(\ref{note1})} pour $i\in\{1,2\}$.
 
 Comme ci-dessus, en rempla\c{c}ant les $f_i$ par des rev\^etements ramifi\'es bien choisis, on peut obtenir deux familles stables qui ont une fibre singuli\`ere en commun et dont les fibres g\'en\'erales, lisses, ne peuvent \^etre membres d'une m\^eme famille lisse de base irr\'eductible.
Il s'agit donc d'un exemple o\`u deux composantes irr\'eductibles distinctes de l'espace des modules des vari\'et\'es stables s'intersectent. Ce ph\'enom\`ene n'appara\^it pas en dimension $1$. Signalons que Horikawa~\cite[Theorem 3]{Horikawa} a construit de tels exemples pour lesquels la fibre sp\'eciale commune aux deux familles est de plus lisse.

\subsection{Espaces de modules de vari\'et\'es stables}

\subsubsection{Existence}
Nous pouvons \`a pr\'esent donner l'\'enonc\'e pr\'ecis d'existence de l'espace de modules des vari\'et\'es stables. La construction de l'espace de modules des courbes stables demandait de fixer le genre de ces courbes. En  dimension sup\'erieure, on doit aussi fixer un invariant discret : la fonction de Hilbert.

\begin{defi}
La \textbf{fonction de Hilbert} $F:\Z\to\Z$ d'une vari\'et\'e stable $X$ est 
$$F(n):=\chi(X,\omega_X^{[n]}).$$
\end{defi}

Comme $\omega_X$ n'est pas inversible en g\'en\'eral, la fonction de Hilbert de $X$ peut ne pas \^etre un polyn\^ome en $n$\footnote{\label{nonpol}Soit $L$ un fibr\'e en droites tr\`es ample sur une surface d'Enriques $S$ et soit $Z=C(S,L)$ le c\^one sur~$S$ dans le plongement induit par $L$. Soit $X\to S$ un rev\^etement double ramifi\'e le long d'une section lisse de $L^{\otimes 4}$. Nous laissons au lecteur le soin de v\'erifier que $X$ est stable et que sa fonction de Hilbert n'est pas un polyn\^ome.}.
L'hypoth\`ese de platitude dans la d\'efinition~\ref{deffam} montre que cet invariant est localement constant sur la base d'une famille stable.

\begin{theo}
\label{thedm}
Soit $F:\Z\to\Z$ une fonction. La cat\'egorie fibr\'ee en groupo\"ides 
\begin{equation}
\label{fonctpoints}
B\mapsto\{\hspace{.1em}\textrm{familles stables }f:\sX\to B\textrm{ dont les fibres ont fonction de Hilbert $F$}\}
\end{equation}
sur la cat\'egorie des $k$-sch\'emas est un champ de Deligne-Mumford $\overline{\mathcal{M}}_F$ propre sur $k$ admettant un espace de modules grossier projectif $\overline{M}_F$.
\end{theo}

La preuve de ce th\'eor\`eme, due \`a de nombreux auteurs, sera esquiss\'ee au \S\ref{consedm}.
Le lecteur qui ne serait pas familier avec les champs~\cite{LMB, Olsson} peut
ne retenir que la seconde partie de son \'enonc\'e. Elle signifie qu'il existe
une vari\'et\'e projective $\overline{M}_F$ sur $k$, et une mani\`ere
d'associer \`a toute famille stable $f:\sX\to B$ dont les fibres ont fonction de Hilbert~$F$
 un morphisme $\psi(f):B\to \overline{M}_F$, qui soit fonctorielle en $B$, de sorte que $(\overline{M}_F,\psi)$ soit universel pour cette propri\'et\'e, et induise une bijection
$$
\left\{
\begin{array}{c}
\mbox{classes d'isomorphisme de vari\'et\'es}\\
\mbox{stables sur $K$ de fonction de Hilbert $F$}
\end{array}
\right\}
\isoto\overline{M}_F(K).
$$
pour toute extension alg\'ebriquement close $K$ de $k$.
Par exemple, le th\'eor\`eme~\ref{thedm} munit l'ensemble des classes d'isomorphisme de vari\'et\'es stables complexes de fonction de Hilbert $F$ d'une structure naturelle de vari\'et\'e projective complexe.

Insistons sur l'importance de la d\'efinition des singularit\'es slc pour la validit\'e de cet \'enonc\'e. Admettre une classe plus large de singularit\'es aurait nui au caract\`ere s\'epar\'e de~$\overline{M}_F$ ; restreindre les singularit\'es autoris\'ees aurait emp\^ech\'e sa propret\'e.

Dans le cas des surfaces, le th\'eor\`eme~\ref{thedm} est connu depuis longtemps, par des travaux de Koll\'ar, Shepherd-Barron et Alexeev~\cite{KSB, Kocomplete, Ale}, \`a deux subtilit\'es pr\`es. D'une part, une structure sch\'ematique sur $\overline{M}_F$ prenant en compte les fonctions nilpotentes, n'a \'et\'e construite rigoureusement que plus tard (voir~\cite{HaKo,KoHH, AbraHa} et \S\ref{ssHH}). D'autre part, la propret\'e des composantes irr\'eductibles de $\overline{M}_F$ param\'etrant g\'en\'eriquement des vari\'et\'es non normales n'a pu \^etre \'etablie que gr\^ace aux techniques de recollement de Koll\'ar 
(voir~\cite{Kosing, Kobook} et \S\ref{nonnormale}).

\subsubsection{G\'eom\'etrie}
La fonction de Hilbert $F(n)=(g-1)(2n-1)$ donne lieu au champ de modules $\overline{\mathcal{M}}_g$ des courbes stables de Deligne et Mumford~\cite{DM} et \`a son espace de modules grossier $\overline{M}_g$. Le champ $\overline{\sM}_g$ est lisse et irr\'eductible, de sorte que $\overline{M}_g$ est normal et irr\'eductible. On a vu \`a la fin du \S\ref{exfam} que ces propri\'et\'es tombaient en d\'efaut en dimension sup\'erieure. Vakil~\cite[Main Theorem M2]{Murphy} a m\^eme d\'emontr\'e que les singularit\'es des vari\'et\'es 
$\overline{M}_F$ peuvent \^etre arbitrairement mauvaises.

La g\'eom\'etrie de $\overline{M}_g$ est aujourd'hui bien comprise et fait l'objet d'une abondante litt\'erature.
A contrario, on dispose de tr\`es peu de descriptions concr\`etes d'espaces de modules non triviaux de vari\'et\'es stables en dimension sup\'erieure (\`a l'exception notable de l'espace de modules des produits de courbes stables~\cite{vOp}). 
On ne sait par exemple pas d\'ecrire l'adh\'erence de l'ouvert param\'etrant des surfaces quintiques dans~$\bP^3_k$~\cite{quintic1, quintic2}. On trouvera dans~\cite{Godeaux} l'\'etat de l'art dans le cas des surfaces de Godeaux.

Comme la fonction de Hilbert d'une vari\'et\'e stable lisse est polynomiale, les vari\'et\'es stables dont la fonction de Hilbert n'est pas polynomiale, comme dans la note de bas de page \eqref{nonpol}, donnent lieu \`a des composantes connexes de l'espace de modules qui ne param\`etrent que des vari\'et\'es singuli\`eres.

Soit enfin $M$ une composante connexe de $\overline{M}_F$. On sait que si l'une des vari\'et\'es que $M$ param\`etre
v\'erifie la condition $S_k$ de Serre, alors toutes ont cette propri\'et\'e~\cite[Corollary 1.3]{lcdB}. Par cons\'equent, si l'une d'entre elles est Cohen-Macaulay (par exemple : lisse), toutes sont Cohen-Macaulay. Il est malgr\'e tout utile de consid\'erer aussi des vari\'et\'es stables qui ne sont pas Cohen-Macaulay ; on en verra une raison au \S\ref{stableslc}.

\subsubsection{Variantes}

De nombreuses variantes des espaces de modules de vari\'et\'es stables sont utiles et ont \'et\'e \'etudi\'ees. Tout d'abord, il est naturel de consid\'erer plut\^ ot des espaces de modules de paires stables, qui g\'en\'eralisent en dimension sup\'erieure les espaces de modules de courbes stables point\'ees. Ce sujet a \'et\'e d\'evelopp\'e dans~\cite{Hassett, Hacking, Alelim, KoPa, Kopaires} et le livre~\cite{Kobook} en fait une \'etude approfondie.

Il est \'egalement int\'eressant de construire des espaces de modules de morphismes stables \`a valeurs dans une vari\'et\'e fix\'ee. Quand la source du morphisme est une courbe, ces espaces ont \'et\'e introduits par Kontsevich (voir~\cite{FuPa}), et on pourra consulter~\cite{Alemod, DR} en dimension sup\'erieure.

Les r\'esultats en caract\'eristique positive sont limit\'es. L'article~\cite{Pata} contient le meilleur \'enonc\'e connu : sur un corps de caract\'eristique $p\geqslant 7$, l'espace de modules des surfaces stables existe comme espace alg\'ebrique s\'epar\'e et ses sous-espaces propres sont projectifs.

\subsection{Outils pour l'\'etude des singularit\'es slc}

  Pour obtenir des compactifications modulaires $\overline{M}_F$ des espaces de modules de vari\'et\'es projectives lisses canoniquement polaris\'ees, nous avons d\^u autoriser des vari\'et\'es \`a singularit\'es slc. Que ce soit pour construire ces compactifications ou pour d'\'eventuelles applications de leur existence, il est important d'\'etudier cette classe de singularit\'es. Il s'av\`ere qu'elles ont des propri\'et\'es remarquables ; nous en d\'ecrivons ici quelques-unes.

\subsubsection{Adjonction}
\label{paradj}

Soit $(X,\Delta+B)$ une paire dans laquelle le diviseur de Weil $B$ est affect\'e d'un coefficient~$1$. On suppose que $X$ est normale, purement de dimension~$d$, et qu'il existe un entier $m>0$ tel que $\omega_X^{[m]}(m\Delta+mB)$ est inversible.
Consid\'erons la normalisation $\nu:\widetilde{B}\to B$ de $B$
et soit $U\subset X$ le plus gros ouvert disjoint de $\Delta$ le long duquel $X$ et $B$ sont tous deux r\'eguliers. L'isomorphisme canonique $\omega_X(B)|_{B\cap U}\isoto \omega_{B\cap U}$ donn\'e par le r\'esidu des formes diff\'erentielles induit un isomorphisme
\begin{equation}
\label{differente}
\omega_X^{[m]}(m\Delta+mB)|_{\widetilde{B}}\isoto\omega_{\widetilde{B}}^{[m]}(m\Diff_{\widetilde{B}}(\Delta)),
\end{equation}
o\`u $\Diff_{\widetilde{B}}(\Delta)$ est un $\Q$-diviseur de Weil  sur $\widetilde{B}$ uniquement d\'etermin\'e et ind\'ependant de~$m$ : c'est la \textbf{diff\'erente} de $\Delta$ sur $\widetilde{B}$ (voir~\cite[Definition 4.2]{Kosing}). 

Dans de nombreuses situations, par exemple dans le cadre d'une r\'ecurrence sur la dimension, il est utile de ramener l'\'etude de $(X,\Delta+B)$ \`a celle de $(\widetilde{B},\Diff_{\widetilde{B}}(\Delta))$. Le th\'eor\`eme~\ref{invadj}, d\^u \`a Kawakita~\cite{Kawakita}, et qui fait suite \`a des travaux de Shokurov~\cite{Shokurov} et de Koll\'ar~\cite[\S 17]{Koflips}, est un outil pr\'ecieux pour ce type d'arguments.

\begin{theo}
\label{invadj}
Les conditions suivantes sont \'equivalentes:
\begin{enumerate}[(i)]
\item La paire $(X,\Delta+B)$ est lc dans un voisinage de $B$.
\item La paire $(\widetilde{B}, \Diff_{\widetilde{B}}(\Delta))$ est lc.
\end{enumerate}
\end{theo}

L'implication (i)$\implies$(ii), dite \textit{adjonction}, est facile. C'est l'implication r\'eciproque (ii)$\implies $(i), dite \textit{inversion de l'adjonction}, qui est d\'elicate. Sa preuve repose de mani\`ere essentielle sur le th\'eor\`eme d'annulation de Kawamata-Viehweg.

Par le biais de l'\'equivalence \eqref{slclc}, on peut d\'eduire du th\'eor\`eme~\ref{invadj} des \'enonc\'es portant sur les singularit\'es slc (voir~\cite[Lemma 2.10, Corollary 2.11]{Patakfibre}).

\subsubsection{Propri\'et\'es cohomologiques} La premi\` ere indication que les classes de singularit\'es que nous consid\'erons ont de bonnes propri\'et\'es cohomologiques a \'et\'e le th\'eor\`eme d'Elkik~\cite{Elkik} selon lequel les singularit\'es canoniques sont rationnelles. Ce r\'esultat reste valide plus g\'en\'eralement pour les singularit\'es klt~\cite[Theorem 5.22]{KM}.

\begin{theo}
\label{rationnelles}
Les singularit\'es klt sont rationnelles.
\end{theo}

On en d\'eduit que les singularit\'es klt sont Cohen-Macaulay~\cite[Theorem 5.10]{KM}.

Malheureusement, les singularit\'es lc ne sont pas toujours rationnelles, ni m\^eme Cohen-Macaulay. Par exemple, un c\^one sur une surface ab\'elienne est lc mais pas $S_3$~\cite[Example 3.6]{Kosing}. Il est donc n\'ecessaire de trouver un substitut \`a la rationalit\'e, qui s'applique aux vari\'et\'es lc (ou plus g\'en\'eralement slc). 
Koll\'ar et Kov\'acs ont montr\'e que les singularit\'es Du Bois~\cite[\S 6]{Kosing} remplissent ce r\^ole (voir~\cite{lcdB},~\cite[\S 6.2]{Kosing}).

\begin{theo}
Les singularit\'es slc sont Du Bois.
\end{theo}

Une cons\'equence concr\`ete de cet \'enonc\'e est le fait que si $f:\sX\to B$ est une famille stable, les fonctions $b\mapsto h^i(\sX_b,\sO_{\sX_b})$ sont localement constantes sur $B$~\cite[Th\'eor\`eme~4.6]{DuBois}.
Nous n'utiliserons pas les singularit\'es Du Bois dans la suite de ce texte.

En revanche, nous devrons savoir contr\^oler pr\'ecis\'ement le d\'efaut de la propri\'et\'e $S_3$ des singularit\'es slc.
Nous utiliserons \`a cet effet un r\'esultat d'Alexeev~\cite[Lemma 3.2]{Alelim}, \'etendu dans~\cite[Theorem 7.20]{Kosing}. On dit qu'une sous-vari\'et\'e int\`egre d'une vari\'et\'e~$X$ \`a singularit\'es slc est un \textbf{centre log canonique} de $X$ si c'est l'image d'un diviseur au-dessus de $X$ dont la discr\'epance est \'egale \`a $-1$.

\begin{theo}
\label{S3}
Soit $X$ une vari\'et\'e slc. Si $x\in X$ n'est pas le point g\'en\'erique d'un centre log canonique de $X$, on a $\prof(\omega^{[n]}_{X,x})\geqslant \min(3, \dim(\sO_{X,x}))$ pour tout $n\in\Z$.
\end{theo}

Les cas $n=0$ et $n=1$ sont explicit\'es dans~\cite[Corollaries 7.21 and 7.22]{Kosing}, et le cas g\'en\'eral se prouve de la m\^eme mani\`ere\footnote{On travaille localement au voisinage de $x$ et on applique~\cite[Theorem 7.20]{Kosing} avec $\Delta=\Delta'=0$ \`a un diviseur de Weil $D\subset X$ tel que $\sO_X(-D)\simeq\omega^{[n]}_X$.}.

\section{Le th\'eor\`eme de r\'eduction stable}
\label{secredst}

Nous expliquons dans cette section la preuve du th\'eor\`eme~\ref{redstable}.
On en consid\`ere plut\^ot une variante locale sur un anneau de valuation discr\`ete $R$ de corps de fonctions $K$. On note $T=\Spec(R)$ son spectre, de point ferm\'e $t$ et de point g\'en\'erique $\eta$.

\begin{theo}
\label{redstable2}
Soit $X$ une vari\'et\'e stable sur $K$.  Il existe une extension finie d'anneaux de valuations discr\`etes $R\subset R'$  de corps de fonctions $K\subset K'$ et une famille stable
 $f:\mathcal{X}\to \Spec(R')$ telle que $\mathcal{X}_{K'}\simeq X_{K'}$. Si $R'$ est fix\'e, cette famille est unique.
 \end{theo}
Que le th\'eor\`eme~\ref{redstable2} implique le th\'eor\`eme de r\'eduction stable sous sa forme globale \'enonc\'ee au th\'eor\`eme~\ref{redstable} est standard.

\begin{proof}[Preuve du th\'eor\`eme~\ref{redstable}]
  Soit $f:\sX\to B$ un morphisme propre de base une courbe int\`egre de corps
  de fonctions $K$.  Si $\sX_{\eta}$ est stable, la famille $f$ est stable
  au-dessus d'un ouvert dense $U\subset B$ (par exemple, par les arguments des
  \S\S\ref{deminormal}--\ref{ouvertslc}). Pour tout $b\in B\setminus U$, le
  th\'eor\`eme~\ref{redstable2} appliqu\'e \`a l'anneau de valuation
  discr\`ete $R_b:=\sO_{B,b}$ fournit une extension finie $R'_b$ de $R_b$ de
  corps de fonctions $K'_b$ telle que $\sX_{K'_b}$ ait un mod\`ele stable sur
  $R'_b$. Soit $K'$ une extension galoisienne de $K$ dans laquelle tous les
  $K'_b$ se plongent et soit $\pi:B'\to B$ la normalisation de $B$ dans
  $K'$. Par construction, la vari\'et\'e $\sX_{K'}$ poss\`ede un mod\`ele
  stable au voisinage de tout point de $B'$. Ces mod\`eles locaux se recollent
  par unicit\'e.
\end{proof}

Le th\'eor\`eme~\ref{redstable2} est d\^u \`a Hacon et Xu~\cite{HX} et Koll\'ar~\cite{Kosing, Kobook}.
C'est ce th\'eor\`eme qui nous permettra de v\'erifier la propret\'e du champ de modules des vari\'et\'es stables (voir \S\ref{champ}). Les r\'esultats ant\'erieurs de Birkar, Cascini, Hacon et McKernan~\cite{BCHM} auraient cependant suffi \`a d\'emontrer la propret\'e des composantes irr\'eductibles de $\overline{\sM}_F$ qui param\`etrent g\'en\'eriquement des vari\'et\'es lisses.

\medskip

La preuve du th\'eor\`eme~\ref{redstable2} repose sur le point de vue selon lequel les vari\'et\'es stables doivent \^etre consid\'er\'ees comme pluricanoniquement plong\'ees. Plus pr\'ecis\'ement, si $X$ est une vari\'et\'e stable et si $m>0$ est tel que 
$\omega_X^{[m]}$ est inversible, on peut reconstruire~$X$ \`a partir de son alg\`ebre $m$-canonique par la formule $X\simeq\Proj\bigoplus_{l \geqslant 0}H^0(X,\omega_X^{[lm]})$. L'existence comme l'unicit\'e des familles stables dans le th\'eor\`eme~\ref{redstable2} seront obtenues par le biais de ces alg\`ebres $m$-canoniques.

\subsection{Unicit\'e}
\label{unicite}

  Montrons la propri\'et\'e d'unicit\'e dans le th\'eor\`eme~\ref{redstable2}. La preuve donn\'ee dans~\cite[Proposition 6]{Kosurvey} quand les $f_i$ sont lisses s'\'etend au cas g\'en\'eral~\cite{Kobook}. Commen\c{c}ons par d\'emontrer un lemme que nous utiliserons \`a plusieurs reprises.
 
\begin{lemm}
\label{XXt}
Soit $f:\sX\to T$ un morphisme propre et plat dont les fibres satisfont les conditions (i)-(iv) de la d\'efinition~\ref{defslc}. Soit $m>0$ un entier tel que $\omega_{\sX/T}^{[m]}$ soit inversible. Si la vari\'et\'e $\sX_t$ est \`a singularit\'es slc, la paire $(\sX,\sX_t)$ est \`a singularit\'es slc. 
\end{lemm}

\begin{proof}[Preuve]
  On consid\`ere le diagramme commutatif
  \begin{equation*}
    \begin{aligned}
      \xymatrix @R=0.4cm {
        \widetilde{\sX_t}\ar^{\tilde{\iota}}[r]\ar^{\nu_t}[d] & \widetilde{\sX} \ar^{\nu}[d] \\
        \sX_t\ar^{\iota}[r]&\sX , }
    \end{aligned}
  \end{equation*}
  o\`u $\nu$ et $\nu_t$ sont les normalisations de $\sX$ et $\sX_t$, et o\`u
  $\iota$ et $\tilde{\iota}$ sont les morphismes naturels. On note
  $\Gamma\subset \widetilde{\sX}$ et $\Delta\subset \widetilde{\sX_t}$ les
  conducteurs de $\sX$ et $\sX_t$ (voir \S\ref{paires}). Par d\'efinition des
  conducteurs et de la diff\'erente (voir~\eqref{conducteur} et~\eqref{differente}), on dispose d'isomorphismes naturels
$$\omega_{\widetilde{\sX_t}}^{[m]}(m\Delta)\simeq \nu_t^*\omega_{\sX_t}^{[m]}\simeq \nu_t^*\iota^*\omega_{\sX}^{[m]}(m\sX_t)\simeq\tilde{\iota}^*\omega_{\widetilde{\sX}}^{[m]}(m(\widetilde{\sX})_t+m\Gamma)\simeq \omega_{\widetilde{\sX_t}}^{[m]}(m\Diff_{\widetilde{\sX_t}}(\Gamma)).$$
La compos\'ee de ces isomorphismes \'etant l'identit\'e de
$\omega_{\widetilde{\sX_t}}^{[m]}$ aux points g\'en\'eriques de
$\widetilde{\sX_t}$, il suit que
$\Delta=\Diff_{\widetilde{\sX_t}}(\Gamma)$. Comme $\sX_t$ est slc, la paire
$(\widetilde{\sX_t},\Delta)$ est lc par~\eqref{slclc}, donc la paire
$(\widetilde{\sX},\Gamma+(\widetilde{\sX})_t)$ est lc par inversion de
l'adjonction (th\'eor\`eme~\ref{invadj}).  On d\'eduit que $(\sX,\sX_t)$ est
slc par~\eqref{slclc}, qui s'adapte imm\'ediatement au cas des paires.
\end{proof}

 Nous pouvons \`a pr\'esent d\'emontrer la propri\'et\'e d'unicit\'e dans le th\'eor\`eme~\ref{redstable2}.
 Soient $f_1:\mathcal{X}_1\to T$ et $f_2:\mathcal{X}_2\to T$ des familles stables et 
$\phi_{\eta}:\sX_{1,\eta}\isoto\sX_{2,\eta}$ un isomorphisme. On souhaite d\'emontrer que $\phi_{\eta}$ s'\'etend en un isomorphisme $\phi:\sX_1\isoto\sX_2$.

Pour cela, notons $\sX\subset\sX_1\times_T\sX_2$ l'adh\'erence du graphe de $\phi_{\eta}$.
Soit $\sY\to\sX$ une modification $S_2$ de $\sX$ qui est un isomorphisme au-dessus de $\sX_\eta$, telle que les composantes irr\'eductibles du lieu non normal de $\sY$ dominent toutes $T$\footnote{Pour construire $\sY$, on note $Z\subset\sX$ l'union des composantes irr\'eductibles du lieu non normal de $\sY$ qui dominent $T$, on remarque que les faisceaux $\sO_{\sX_\eta}$ sur $\sX_{\eta}$ et $\widetilde{\sO_{\sX\setminus Z}}$ sur $\sX\setminus Z$ se recollent en un faisceau d'alg\`ebres coh\'erent $\sA$ sur $\sX\setminus Z_t$ et on d\'efinit $\sY:=\Spec_{\sO_{\sX}}(j_*\sA)$, o\`u $j:\sX\setminus Z_t\hookrightarrow\sX$ est l'inclusion et o\`u $j_*\sA$ est un faisceau d'alg\`ebres coh\'erent $S_2$ par~\cite[Propositions  5.11.1 et 5.10.10]{EGA42}.}. Notons $g_i:\sY\to\sX_i$ les projections naturelles, et choisissons un entier $m>0$ tel que les $\omega^{[m]}_{\sX_i/T}$ soient inversibles.

Si $i\in\{1,2\}$, la paire $(\sX_i,(\sX_i)_t)$ est slc par le lemme~\ref{XXt}.
Pour $l\geqslant 0$, on dispose par~\eqref{discrepeq} d'un isomorphisme
$\omega^{[lm]}_{\sY}\isoto g_i^*\omega_{\sX_i}^{[lm]}(\sum_E lm \cdot a_{E}(\sX_i)E)$, o\`u la somme porte sur les diviseurs $g_i$-exceptionnels $E$ de $\sY_t$, et o\`u les $a_{E}(\sX_i)$ sont les discr\'epances de $\sX_i$.
 On en d\'eduit un isomorphisme $\omega^{[lm]}_{\sY}(lm\sY_t)\isoto (g_i^*\omega_{\sX_i}^{[lm]}(lm(\sX_i)_t))(\sum_E lm \cdot a_{E}(\sX_i)E)$. En le comparant \`a l'isomorphisme~\eqref{logdiscrepeq} d\'efinissant les discr\'epances de 
$(\sX_i,(\sX_i)_t)$, on voit que $a_{E}(\sX_i)=a_{E}(\sX_i,(\sX_i)_t)+b_E$ o\`u $b_E$ est la multiplicit\'e de $E$ dans $\sY_t$. On a donc $a_{E}(\sX_i)\geqslant -1+1=0$ car $(\sX_i,(\sX_i)_t)$ est slc. On en d\'eduit les \'egalit\'es
\begin{equation}
\label{discreptrick}
H^0(\sX_i, \omega_{\sX_i}^{[lm]})=H^0(\sY, g_i^*\omega_{\sX_i}^{[lm]}(\sum_E lm \cdot a_{E}(\sX_i)E))=H^0(\sY, \omega_{\sY}^{[lm]}),
\end{equation}
o\`u seule la premi\`ere \'egalit\'e est \`a justifier. Que le membre de gauche soit inclus dans celui de droite est une cons\'equence de la positivit\'e des  $a_{E}(\sX_i)$. 
Pour voir l'autre inclusion, notons $U_i\subset\sX_i$ l'ouvert au-dessus duquel $g_i$ est un isomorphisme. Si $\sigma\in H^0(\sY, g_i^*\omega_{\sX_i}^{[lm]}(\sum_E lm \cdot a_{E}(\sX_i)E))$, la restriction $\sigma|_{U_i}\in H^0(U_i,\omega_{U_i}^{[lm]})$ se rel\`eve \`a $H^0(\sX_i,\omega_{\sX_i}^{[lm]})$ par propri\'et\'e $S_2$ de $\sX_i$, car $\sX_i\setminus U_i$ a codimension $\geqslant 2$ dans $\sX_i$. 

On conclut en d\'efinissant $\phi$ par la cha\^ine d'isomorphismes naturels suivante, 
o\`u nous utilisons l'amplitude de $\omega_{\sX_1}^{[m]}$ et $\omega_{\sX_2}^{[m]}$ :
$$\sX_1\hspace{-.1em}\simeq \hspace{-.1em}\Proj_T\bigoplus _{l\geqslant 0}\hspace{-.1em}H^0(\sX_1, \omega_{\sX_1}^{[lm]})\hspace{-.1em}\simeq\hspace{-.1em} \Proj_T\bigoplus _{l\geqslant 0} \hspace{-.1em}H^0(\sY, \omega_{\sY}^{[lm]})\hspace{-.1em}\simeq \hspace{-.1em}\Proj_T\bigoplus _{l\geqslant 0} \hspace{-.1em}H^0(\sX_2, \omega_{\sX_2}^{[lm]})\hspace{-.1em}\simeq \hspace{-.1em}\sX_2.$$

\medskip

La preuve ci-dessus fait clairement appara\^itre le r\^ole de la condition (v) sur les discr\'epances dans la d\'efinition~\ref{defslc} : c'est elle qui permet d'identifier les alg\`ebres $m$-canoniques de $\sX_1$ et $\sX_2$, donc $\sX_1$ et $\sX_2$.

\subsection{Fibre g\'en\'erique normale}
\label{pargennormale}

Passons \`a l'assertion d'existence dans le th\'eor\`eme~\ref{redstable2}. On suppose dans ce paragraphe que la vari\'et\'e stable $X$ sur $K$ est normale, donc lc. Ce cas particulier crucial est d\^u \`a Hacon et Xu~\cite[Corollary 1.5]{HX}. Il repose sur le programme des mod\`eles minimaux par le biais de~\cite[Theorem 1.1]{HX} que nous discuterons plus au \S\ref{finitude}.

\subsubsection{Construction du mod\`ele stable}
\label{modelestable}

Soit $\overline{f}:\overline{\sX}\to T$ un morphisme projectif et plat tel que $\overline{\sX}_\eta\simeq X$.
Par le th\'eor\`eme de r\'eduction semi-stable de Kempf, Knudsen, Mumford et Saint-Donat~\cite[p.~198]{KKMS} (voir le th\'eor\`eme~\ref{KKMS}) sous la forme plus pr\'ecise \'enonc\'ee dans~\cite[Theorem 7.17]{KM}, on peut supposer, quitte \`a remplacer $R$ par une extension finie et $\overline{\sX}$ par le changement de base normalis\'e, qu'il existe une modification $\mu:\sY\to \overline{\sX}$
telle que $\sY$ soit r\'egulier et $\sY_t$ soit r\'eduit, et telle que si l'on note $\Delta\subset\sY$ l'adh\'erence du lieu exceptionnel
de $\mu_{\eta}$, le diviseur $\Delta+\sY_t$ est \`a croisements normaux stricts dans $\sY$. Cette application du th\'eor\`eme de r\'eduction semi-stable est le seul moment o\`u, dans la preuve du th\'eor\`eme~\ref{redstable2}, on doit modifier l'anneau de valuation discr\`ete de base.

Soit $m>0$ un entier tel que $\omega_{X}^{[m]}$ soit inversible.  Comme $X$ est \`a singularit\'es lc, on a $a_{\Delta_i}(X)\geqslant -1$ pour toute composante irr\'eductible $\Delta_i$ de $\Delta$. On d\'eduit, par un argument analogue \`a celui qui a d\'emontr\'e~\eqref{discreptrick}, que $H^0(X, \omega_X^{[lm]})=H^0(\sY_{\eta},\omega_{\sY_{\eta}}^{[lm]}(lm\Delta_{\eta}))$ pour tout $l\geqslant 0$. 
Comme $\omega_X^{[m]}$ est ample, l'alg\`ebre $m$-canonique 
$$A_K:=\bigoplus_{l\geqslant 0} H^0(\sY_\eta,\omega_{\sY_\eta}^{[lm]}(lm\Delta_\eta))$$
 est de type fini sur $K$ et $\Proj(A_K)\simeq X$.
Par~\cite[Theorem 1.1]{HX} (voir le th\'eor\`eme~\ref{thHX}) appliqu\'e au morphisme $\overline{f}\circ\mu:\sY\to T$, l'alg\`ebre $m$-canonique
$$A:=\bigoplus_{l\geqslant 0} H^0(\sY,\omega_{\sY}^{[lm]}(lm\Delta+lm\sY_t))$$
 de $(\sY,\Delta+\sY_t)$ est de type fini sur $R$. 
On peut donc former le mod\`ele canonique relatif $f:\sX:=\Proj_T A\to T$ de $(\sY,\Delta+\sY_t)$ au-dessus de $T$. Notons $\phi:\sY\dashrightarrow\sX$ l'application rationnelle naturelle.
On affirme que $f:\sX\to T$ est la famille stable recherch\'ee. 

Comme $\sX_\eta\simeq X$, il reste \`a d\'emontrer que $f$ est stable. C'est le but des \S\S~\ref{total}--\ref{famille}.

\subsubsection{\'Etude de l'espace total}
\label{total}
On commence par \'etudier l'espace total $\sX$ du morphisme $f:\sX\to T$. Pour ce faire, on s'appuie sur des propri\'et\'es \'el\'ementaires des mod\`eles canoniques relatifs, rassembl\'ees dans~\cite[Theorem 1.26]{Kosing}.

On montre ainsi que $\sX$ est normal, que l'application birationnelle $\phi:\sY\dashrightarrow\sX$ est une contraction rationnelle\footnote{Cela signifie que son inverse ne contracte pas de diviseurs.}
 et que, quitte \`a remplacer $m$ par un multiple, le faisceau $\omega^{[m]}_{\sX}(m\phi_*(\Delta+\sY_t))$ est inversible et $f$-ample. Comme $X\simeq \sX_{\eta}$ et comme $\Delta_\eta$ est contract\'e par $\mu_\eta$,  on a $\phi_*\Delta=0$.  La multiplication par une uniformisante de $R$ induit un isomorphisme $\sO_{\sX}\isoto\sO_{\sX}(\sX_t)=\sO_{\sX}(\phi_*\sY_t)$ ; on voit donc que $\omega^{[m]}_{\sX}$ est inversible et $f$-ample.
Enfin, une derni\`ere assertion de~\cite[Theorem 1.26]{Kosing} est que comme la paire $(\sY,\Delta+\sY_t)$ est \`a singularit\'es lc, il en va de m\^eme pour $(\sX,\phi_*(\Delta+\sY_t))=(\sX,\sX_t)$.

\subsubsection{\'Etude de la fibre sp\'eciale}
\label{speciale}

 Il est temps de d\'emontrer que la fibre sp\'eciale $\sX_t$ de $f$ est stable. Il ne reste plus qu'\`a voir que ses singularit\'es sont slc.

 Comme la fibre sp\'eciale $\sY_t$ de $\sY$ est r\'eduite et que $\phi$ est une contraction birationnelle, on voit que $\sX_t$ est g\'en\'eriquement r\'eduite. De plus, $\sX_t$ est $S_1$ comme diviseur de Cartier dans $\sX$ qui est normal donc $S_2$. Ces deux faits combin\'es montrent exactement que $\sX_t$ est r\'eduite. Nous avons v\'erifi\'e la condition (i) de la d\'efinition~\ref{defslc}.

La condition (ii) selon laquelle $\sX_t$ est au plus nodale en codimension $1$ r\'esulte du fait que $(\sX,\sX_t)$ est \`a singularit\'es lc et d'une \'etude fine des paires \`a singularit\'es lc en un point de codimension $2$ se situant sur une composante affect\'ee d'un coefficient $1$ du bord~\cite[Corollary 2.32]{Kosing}. Qu'une telle \'etude soit possible est \`a rapprocher du fait que l'on sache classifier les singularit\'es lc des surfaces~\cite[\S 4.1]{KM}.

Comme $(\sX,\sX_t)$ est \`a singularit\'es lc, il en va a fortiori de m\^eme pour $\sX$. Soit $E$ un diviseur au-dessus de $\sX$ dont l'image dans $\sX$ se situe sur la fibre sp\'eciale $\sX_t$. L'in\'egalit\'e $a_E(\sX)\geqslant a_E(\sX,\sX_t)+1\geqslant 0$ montre que $\sX_t$ ne contient aucun centre log canonique de $\sX$. Il suit du th\'eor\`eme~\ref{S3} que $\prof(\sO_{\sX,x})\geqslant \min(3, \dim(\sO_{\sX,x}))$ pour tout $x\in\sX_t$. Comme $\sX_t$ est un diviseur de Cartier dans $\sX$, on d\'eduit que $\prof(\sO_{\sX_t,x})\geqslant \min(2, \dim(\sO_{\sX_t,x}))$ pour tout $x\in\sX_t$, ce qui est la condition (iii).

Le faisceau $\omega_{\sX}^{[m]}|_{\sX_t}$ est inversible car $\omega_{\sX}^{[m]}$ l'est. Il co\"incide donc avec $\omega_{\sX_t}^{[m]}$ car ces deux faisceaux sont~$S_2$ et isomorphes en codimension $1$. Ceci d\'emontre que $\omega_{\sX_t}^{[m]}$ est inversible, donc que la condition (iv) est satisfaite.

Enfin, la paire $(\sX,\sX_t)$ \'etant \`a singularit\'es lc, il en va de m\^eme pour $(\widetilde{\sX_t},\Diff_{\widetilde{\sX_t}}(0))$ par adjonction (th\'eor\`eme~\ref{invadj}). Nous avons vu dans la preuve du lemme~\ref{XXt} que $\Diff_{\widetilde{\sX_t}}(0)$ est le conducteur de $\sX_t$. On d\'eduit donc de~\eqref{slclc} que $\sX_t$ est slc.

\subsubsection{Stabilit\'e de la famille}
\label{famille}

Il reste enfin \`a d\'emontrer que la famille $f:\sX\to T$ est stable au sens de la d\'efinition~\ref{deffam}. Le morphisme $f$ est plat puisque $\sX$ est r\'eduit et que toutes ses composantes irr\'eductibles dominent $T$. Nous avons d\'ej\`a montr\'e que ses fibres sont  \`a singularit\'es slc.

Soit $n\in\Z$. Le faisceau $\omega_{\sX/T}^{[n]}$ est plat car il est $S_1$ par construction et car les composantes irr\'eductibles de son support dominent $T$. Nous avons d\'ej\`a vu au \S\ref{speciale} que $\sX_t$ ne contient aucun centre log canonique de $\sX$. Le th\'eor\`eme~\ref{S3} montre donc que $\prof(\omega^{[n]}_{\sX/T,x})\geqslant \min(3, \dim(\sO_{\sX,x}))$ pour tout $x\in\sX_t$. Comme $\sX_t$ est un diviseur de Cartier dans $\sX$, on d\'eduit que $\omega^{[n]}_{\sX/T}|_{\sX_t}$ est $S_2$. Les deux faisceaux $\omega^{[n]}_{\sX/T}|_{\sX_t}$ et $\omega^{[n]}_{\sX_t}$ sont~$S_2$ et isomorphes en codimension $1$ ; ils co\"incident donc. Cela entra\^ine la stabilit\'e de la famille $f$ et ach\`eve la preuve du th\'eor\`eme~\ref{redstable2} quand $X$ est normale.

\subsection{Fibre g\'en\'erique non normale}
\label{nonnormale}

Expliquons maintenant l'\'enonc\'e d'existence du th\'eor\`eme~\ref{redstable2} dans le cas g\'en\'eral.

\subsubsection{Normalisation}
\label{techglue}

La preuve, due \`a Koll\'ar~\cite{Kobook}, proc\`ede  par r\'eduction au cas normal.
On applique le th\'eor\`eme de r\'eduction stable \`a la normalisation de $X$ qui est justiciable des arguments du \S\ref{pargennormale}, et on construit $f:\sX\to T$
 en \textit{recollant} le mod\`ele stable obtenu le long de lui-m\^eme pour faire appara\^itre les singularit\'es non normales requises. L'\'etape de recollement est surprenamment
d\'elicate \`a mettre en \oe{}uvre et constitue le c\oe{}ur du livre~\cite{Kosing}. Expliquons son principe.

Soit $X$ une vari\'et\'e \`a singularit\'es slc. Notons $\pi:\widetilde{X}\to X$ sa normalisation et $\Gamma$ son conducteur. On sait par~\eqref{slclc} que la paire $(\widetilde{X},\Gamma)$ est lc.
 Le morphisme $\pi|_{\Gamma}:\Gamma\to\pi(\Gamma)$ est de degr\'e deux au-dessus de l'ouvert dense de $\pi(\Gamma)$ le long duquel $X$ est \`a croisements normaux doubles. On en d\'eduit une involution rationnelle $\tau:\Gamma\dashrightarrow\Gamma$, qui s'\'etend en une involution r\'eguli\`ere $\tau:\widetilde{\Gamma}\to\widetilde{\Gamma}$ g\'en\'eriquement sans point fixe de la normalisation $\nu:\widetilde{\Gamma}\to\Gamma$ de $\Gamma$. G\'eom\'etriquement, $\tau$ \'echange les deux branches des singularit\'es \`a croisements normaux doubles de $X$. En comparant l'\'equation~\eqref{differente} d\'efinissant la diff\'erente et son pull-back par $\tau$, on voit que le $\Q$-diviseur $\Diff_{\widetilde{\Gamma}}(0)$ de $\widetilde{\Gamma}$ est $\tau$-invariant.

\medskip

Donnons deux exemples de ces constructions. Si $X=\{x^2=yz^2\}\subset \bA_k^3$ est le parapluie de Whitney, on a $\widetilde{X}=\bA^2_k$, la normalisation $\pi:\widetilde{X}\to X$ est donn\'ee par $(u,v)\mapsto(uv,v^2,u)$, on a $\Gamma=\widetilde{\Gamma}=\{u=0\}\subset \bA_k^2$ et $\tau:v\mapsto -v$, et $\Diff_{\widetilde{\Gamma}}(0)=0$.

Si $X=\{xyz=0\}\subset \bA_k^3$ est le point \`a croisements normaux triples, la normalisation~$\widetilde{X}$ est une union disjointe de trois espaces affines de dimension $2$, le conducteur $\Gamma$ est l'union de leurs axes de coordonn\'ees, de sorte que $\widetilde{\Gamma}$ est une union de six droites, et on calcule que $\Diff_{\widetilde{\Gamma}}(0)\subset\widetilde{\Gamma}$ est la r\'eunion de leurs origines. L'involution $\tau:\widetilde{\Gamma}\to\widetilde{\Gamma}$ \'echange ces droites deux par deux. Dans cet exemple, l'involution rationnelle $\tau:\Gamma\dashrightarrow\Gamma$ n'est pas r\'eguli\`ere en les trois points singuliers de $\Gamma$.

\medskip 

Koll\'ar a remarqu\'e que, sous des hypoth\`eses appropri\'ees, on peut construire la vari\'et\'e~$X$ \`a partir des donn\'ees $(\widetilde{X},\Gamma,\tau)$. Un exemple prototypique (qui n'est pas l'\'enonc\'e pr\'ecis dont on aura besoin pour la preuve du th\'eor\`eme~\ref{redstable2}) est~\cite[Theorem 5.13]{Kosing}.

\begin{theo}
\label{bij}
Les constructions ci-dessus induisent une bijection 
\begin{equation*}
\left\{\hspace{-.5em}
\begin{array}{c}
\mbox{Classes d'isomorphisme}\\
\mbox{de vari\'et\'es stables $X$}
\end{array}
\hspace{-.3em}\right\}\hspace{-.2em}\isoto\hspace{-.2em}\left\{\hspace{-.6em}
\begin{array}{c}
\mbox{Classes d'isomorphisme de paires lc stables}\\
\mbox{ $(\widetilde{X},\Gamma)$ munies d'une involution g\'en\'eriquement }\\
\mbox{ sans point fixe $\tau$ de $(\widetilde{\Gamma},\Diff_{\widetilde{\Gamma}}(0))$}
\end{array}
\hspace{-.5em}\right\}.
\end{equation*}
\end{theo}

\subsubsection{\'Etapes du recollement}
\label{recoller}

Il est facile de voir que l'application du th\'eor\`eme~\ref{bij} est injective~\cite[Proposition 5.3]{Kosing}. C'est sa surjectivit\'e qui est difficile. Le triplet $(\widetilde{X}, \Gamma,\tau)$ \'etant donn\'e, il s'agit de construire $X$ en recollant $\widetilde{X}$ sur elle-m\^eme le long de $\Gamma$ de la mani\`ere indiqu\'ee par $\tau$.
Plut\^ot que d'expliquer la d\'emonstration, dont la structure inductive est complexe, d\'ecrivons les difficult\'es qu'il faut surmonter, qui correspondent aussi aux \'etapes de la preuve du th\'eor\`eme~\ref{bij}.

\medskip

(i) On souhaite construire $X$ comme quotient de $\widetilde{X}$ par la relation d'\'equivalence qui identifie $\nu(x)$ et $\nu(\tau(x))$ pour tout point g\'eom\'etrique $x\in\widetilde{\Gamma}$. Comme le morphisme $\pi:\widetilde{X}\to X$ \`a construire est fini, il faut que la relation d'\'equivalence engendr\'ee par ces relations ait des classes d'\'equivalence finies. Ce n'est pas du tout une \'evidence !

Par exemple, prenons $\widetilde{X}=\bA^3_k$ et $\Gamma=\{xy=0\}$ de sorte que $\widetilde{\Gamma}$ est l'union de deux plans affines de coordonn\'ees respectives $(y_1,z_1)$ et $(x_2,z_2)$ et que $\Diff_{\widetilde{\Gamma}}(0)$ est l'union des deux droites d'\'equations $\{y_1=0\}$ et $\{x_2=0\}$. D\'efinissons une involution $\tau$ \'echangeant ces deux plans par l'\'equation $\tau(y_1,z_1)=(x_2,z_2+1)$.
Pour ces choix de $(\widetilde{X}, \Gamma,\tau)$, on voit que les points $(0,0,n)\in\widetilde{X}$ pour $n\in\Z$ sont tous \'equivalents.

Dans le cadre du th\'eor\`eme~\ref{bij}, ce sont les hypoth\`eses globales de projectivit\'e de $\widetilde{X}$ et d'amplitude de $\omega_{\widetilde{X}}(\Gamma)$ qui assureront la finitude de ces classes d'\'equivalences~\cite[Corollary 5.37]{Kosing}. La preuve de ce fait repose en dernier lieu sur des r\'esultats de finitude pour des groupes d'automorphismes birationnels de paires dont le fibr\'e canonique a des propri\'et\'es de positivit\'e~\cite[Corollary 10.69]{Kosing}.

(ii)  Supposons le probl\`eme d\'ecrit en (i) r\'esolu. On dispose alors d'une relation d'\'equivalence finie sur $\widetilde{X}$ dont on souhaite construire le quotient comme vari\'et\'e alg\'ebrique. Ce serait la vari\'et\'e $X$ recherch\'ee. Il n'est malheureusement pas du tout \'evident que ce soit possible.

Donnons un exemple en suivant~\cite[Example 9.7]{Kosing}.
 Soient $\widetilde{X}$ l'union de deux espaces affines de dimension $3$, de coordonn\'ees respectives $(x_1,y_1,z_1)$ et $(x_2,y_2,z_2)$, et $\Gamma\subset\widetilde{X}$ d\'efini par les \'equations $\{y_i^3=z_i^2\}$ pour $i\in\{1,2\}$. La normalisation $\widetilde{\Gamma}$ de $\Gamma$ est une union de deux plans affines, de coordonn\'ees respectives $(u_1,v_1)$ et $(u_2,v_2)$, et le morphisme $\nu:\widetilde{\Gamma}\to\Gamma$ est donn\'e par $(u_i,v_i)\mapsto (u_i,v_i^2,v_i^3)$. Choisissons pour $\tau$ l'involution \'echangeant ces deux plans, d\'efinie par la formule $\tau(u_1,v_1)=(u_1+v_1,v_1)$.

 On v\'erifie ais\'ement que la relation d'\'equivalence engendr\'ee par $\nu(x)\sim\nu(\tau(x))$ est finie, de sorte que le probl\`eme soulev\'e en (i) n'appara\^it pas. De plus, cette relation d'\'equivalence admet bien un quotient cat\'egorique dans la cat\'egorie des $k$-sch\'emas : le spectre de la sous-$k$-alg\`ebre de $k[x_1,y_1,z_1]\times k[x_2,y_2,z_2]$ engendr\'ee par les id\'eaux $\langle y_1,z_1\rangle$ et $\langle y_2,z_2\rangle$. Cette alg\`ebre n'est pas de type fini sur $k$ (ni m\^eme noeth\'erienne). Le quotient de $\widetilde{X}$ par la relation d'\'equivalence consid\'er\'ee n'est donc pas une vari\'et\'e.

Le probl\`eme avec cet exemple est que la paire $(\widetilde{X},\Gamma)$ n'est pas lc. 
C'est seulement sous l'hypoth\`ese que les singularit\'es de $(\widetilde{X},\Gamma)$ sont lc que Koll\'ar montre l'existence du quotient $X$ recherch\'e~\cite[Theorem 5.32]{Kosing}. 
Cette hypoth\`ese est utilis\'ee de la mani\`ere suivante. La vari\'et\'e $X$ est obtenue par un proc\'ed\'e inductif qui consiste, en simplifiant, \`a d'abord construire les quotients des centres log canoniques de $(\widetilde{X},\Gamma)$, en commen\c{c}ant par ceux qui ont dimension minimale. Pour ce faire, on utilise de mani\`ere essentielle des propri\'et\'es de seminormalit\'e des centres log canoniques~\cite[\S 4.20]{Kosing}, qui permettent en un sens de les manipuler topologiquement. 
Dans l'exemple ci-dessus, c'est le d\'efaut de seminormalit\'e de $\Gamma$ qui pose v\'eritablement probl\`eme.

(iii) Maintenant que la vari\'et\'e $X$ est construite, il faut v\'erifier qu'elle a les propri\'et\'es requises. Si la plupart sont faciles \`a v\'erifier, l'existence d'un entier $m>0$ tel que $\omega_X^{[m]}$ soit inversible est hautement non triviale.
\`A nouveau, illustrons-le sur un exemple.

Soient $\widetilde{X}$ l'union disjointe de trois plans affines de coordonn\'ees $(x_1,y_1)$, $(x_2,y_2)$ et $(x_3,y_3)$, et $\Gamma\subset\widetilde{X}$ le diviseur d\'efini par les \'equations $\{y_1=0\}$, $\{x_2y_2=0\}$ et $\{x_3=0\}$. La normalisation $\widetilde{\Gamma}$ de $\Gamma$ est une union disjointe de quatre droites affines de coordonn\'ees respectives $x_1$, $x_2$, $y_2$ et $y_3$.  Choisissons pour $\tau$ l'involution de $\widetilde{\Gamma}$ \'echangeant les deux premi\`eres droites par la formule $\tau(x_1)=x_2$, et les deux derni\`eres par $\tau(y_2)=y_3$. Aucun des probl\`emes d\'ecrits en (i) et (ii) ne se pose et l'on peut donc consid\'erer la vari\'et\'e~$X$ quotient de $\widetilde{X}$ par la relation d'\'equivalence engendr\'ee par $\nu(x)\sim\nu(\tau(x))$. Avec les notations de~\eqref{cone}, on a $X=C(Y,L)$, o\`u $Y$ est une cha\^ine de trois droites projectives et o\`u le fibr\'e en droites ample $L$ sur $Y$ a degr\'e $1$ sur chacune de ces trois composantes.  On v\'erifie alors en adaptant~\cite[Proposition 3.14 (4)]{Kosing} que $\omega_X^{[m]}$ n'est inversible pour aucun $m>0$.

Pour expliquer cela, remarquons que la diff\'erente $\Diff_{\widetilde{\Gamma}}(0)\subset \widetilde{\Gamma}$ est l'union des origines de la deuxi\`eme et de la troisi\`eme composante de $\widetilde{\Gamma}$. On voit donc que $\tau$ ne pr\'eserve pas $\Diff_{\widetilde{\Gamma}}(0)$. Dans la preuve du th\'eor\`eme~\ref{bij}, c'est l'hypoth\`ese que $\Diff_{\widetilde{\Gamma}}(0)$ soit $\tau$-invariant qui assure l'existence d'un $m>0$ tel que $\omega_X^{[m]}$ est inversible
\cite[Theorem 5.38]{Kosing}. 
Cette hypoth\`ese est utilis\'ee comme suit. Soit $m>0$ tel que $\omega^{[m]}_{\widetilde{X}}(m\Gamma)$ est inversible.
On souhaite descendre $\omega^{[m]}_{\widetilde{X}}(m\Gamma)$ (ou une de ses puissances) en un faisceau inversible sur $X$, isomorphe \`a $\omega_X^{[m]}$ (ou \`a une de ses puissances). Pour ce faire, on raisonne g\'eom\'etriquement en consid\'erant l'espace total $\widetilde{p}:\widetilde{L}\to\widetilde{X}$ du fibr\'e en droites associ\'e \`a $\omega^{[m]}_{\widetilde{X}}(m\Gamma)$ sur $\widetilde{X}$. Notons $\Delta:=\widetilde{p}^{-1}(\Gamma)$. Le fait que la diff\'erente $\Diff_{\widetilde{\Gamma}}(0)$ soit $\tau$-invariante implique que $\tau$~se rel\`eve naturellement en une involution $\sigma$ de la normalisation $\widetilde{\Delta}$ de $\Delta$. On peut alors appliquer les \'etapes (i) et (ii) de la technique de recollement au triplet $(\widetilde{L},\Delta,\sigma)$. On construit de la sorte une vari\'et\'e $p:L\to X$ qu'on v\'erifie \^etre (quitte \`a remplacer $m$ par un multiple) le fibr\'e en droites associ\'e \`a $\omega^{[m]}_X$. Cela implique en particulier que $\omega^{[m]}_X$ est inversible, ce qu'on d\'esirait montrer.

\subsubsection{R\'eduction stable}
\label{stableslc}
Expliquons maintenant, en suivant~\cite{Kobook}, comment la m\'ethode de recollement est utilis\'ee pour d\'emontrer l'assertion d'existence dans le th\'eor\`eme~\ref{redstable2}.

Soit $X$ une vari\'et\'e stable sur $K$. Notons $\widetilde{X}$ sa normalisation, $\Gamma$ son conducteur et $\tau:\widetilde{\Gamma}\to \widetilde{\Gamma}$ l'involution naturelle, qui pr\'eserve $\Diff_{\widetilde{\Gamma}}(0)$.
 Le th\'eor\`eme de r\'eduction stable pour les vari\'et\'es normales (voir \S\ref{pargennormale}), convenablement \'etendu au cas des paires, montre que $(\widetilde{X},\Gamma)$ admet un mod\`ele stable $\widetilde{f}:(\widetilde{\sX},\mathit{\Gamma})\to T$ sur $T$.
Le lemme~\ref{XXt}, adapt\'e au cas des paires, montre que la paire $(\widetilde{\sX},\mathit{\Gamma}+\widetilde{\sX}_t)$ est lc, et on d\'eduit donc de l'adjonction (th\'eor\`eme~\ref{invadj}) que $(\widetilde{\mathit{\Gamma}},\Diff_{\widetilde{\mathit{\Gamma}}}(\widetilde{\sX}_t))=(\widetilde{\mathit{\Gamma}},\Diff_{\widetilde{\mathit{\Gamma}}}(0)+\widetilde{\mathit{\Gamma}}_t)$ est lc. De plus, pour $m>0$ bien choisi, $\omega^{[m]}_{\widetilde{\mathit{\Gamma}}}(m\Diff_{\widetilde{\mathit{\Gamma}}}(0))=\omega^{[m]}_{\widetilde{\sX}}(m\mathit{\Gamma})|_{\widetilde{\mathit{\Gamma}}}$ est un faisceau inversible ample relativement \`a $T$. Argumentant comme aux \S\S\ref{speciale}--\ref{famille}, on voit que 
$(\widetilde{\mathit{\Gamma}},\Diff_{\widetilde{\mathit{\Gamma}}}(0))\to T$ est une famille stable.
Par l'\'enonc\'e d'unicit\'e dans le th\'eor\`eme de r\'eduction stable (voir \S\ref{unicite}), convenablement \'etendu au cas des paires, l'involution $\tau$ sur la fibre g\'en\'erique s'\'etend en une involution encore not\'ee $\tau:(\widetilde{\mathit{\Gamma}},\Diff_{\widetilde{\mathit{\Gamma}}}(0))\to (\widetilde{\mathit{\Gamma}},\Diff_{\widetilde{\mathit{\Gamma}}}(0))$. On applique alors la technique de recollement d\'ecrite au \S\ref{recoller}\footnote{L'\'etape (i) du proc\'ed\'e de recollement est plus facile \`a mettre en \oe{}uvre ici que dans le cadre du th\'eor\`eme~\ref{bij}. En effet, on peut exploiter l'existence du recollement $X$ de $(\widetilde{X},\Gamma,\tau)=(\widetilde{\sX}_\eta,\mathit{\Gamma}_\eta,\tau_\eta)$, et le fait (d\'ej\`a expliqu\'e au \S\ref{speciale}) qu'aucun centre log canonique de $(\widetilde{\sX},\mathit{\Gamma})$ n'est inclus dans la fibre sp\'eciale $\widetilde{\sX}_t$, et appliquer~\cite[Lemma 9.55]{Kosing}. Les \'etapes (ii) et (iii) sont en revanche inchang\'ees.} au triplet $(\widetilde{\sX},\mathit{\Gamma},\tau)$, ce qui donne lieu \`a un morphisme $f:\sX\to T$, dont on v\'erifie qu'elle est la famille stable recherch\'ee.

\medskip

La preuve que nous venons de d\'ecrire ne permet pas de se limiter aux vari\'et\'es stables qui sont Cohen-Macaulay. En effet, la normalisation $\widetilde{X}$ d'une vari\'et\'e stable Cohen-Macaulay $X$ peut ne pas \^etre elle-m\^eme Cohen-Macaulay~\cite[Example 23]{Kosurvey}.

\subsection{Finitude de l'alg\`ebre canonique}
\label{finitude}

Revenons sur le th\'eor\`eme de Hacon et Xu que nous avons utilis\'e au \S\ref{pargennormale}, et qui est un ingr\'edient d\'ecisif de la preuve du th\'eor\`eme~\ref{redstable}. 

Soit $(X,\Delta)$ une paire \`a singularit\'es slc. 
D\'efinissons l'\textbf{alg\`ebre canonique} de $(X,\Delta)$ comme \'etant $A(X,\Delta):=\bigoplus_{\l\geqslant 0} H^0(X,\omega_X^{[l]}(\left \lfloor{l\Delta}\right \rfloor))$, o\`u $\left \lfloor{l\Delta}\right \rfloor$ est le diviseur de Weil sur $X$ obtenu en arrondissant les coefficients de $l\Delta$ \`a l'entier inf\'erieur. On conjecture (voir par exemple ~\cite[Conjecture A]{FG}) la propri\'et\'e suivante.

\begin{conj}
\label{conjfin}
L'alg\`ebre canonique d'une paire projective lc est de type fini.
\end{conj}

Les premiers r\'esultats concernant la conjecture~\ref{conjfin} en dimension arbitraire ont \'et\'e obtenus par Birkar, Cascini, Hacon et McKernan~\cite{BCHM}. Ils la r\'esolvent en particulier pour les paires de type g\'en\'eral\footnote{La paire $(X,\Delta)$ est \textbf{de type g\'en\'eral} s'il existe un entier $m>0$ tel que $\omega_X^{[m]}(m\Delta)$ soit inversible et induise une application rationnelle $X\dashrightarrow\bP^N_k$ qui est birationnelle sur son image.} \`a singularit\'es klt.  Ce travail remarquable a d\'ej\`a fait l'objet d'un expos\'e dans ce s\'eminaire~\cite{Druel} et est \`a la base des d\'eveloppements ult\'erieurs.

De mani\`ere surprenante, la conjecture~\ref{conjfin} tombe en d\'efaut pour les vari\'et\'es slc : Koll\'ar a donn\'e un exemple de surface projective slc qui est de type g\'en\'eral mais dont l'alg\`ebre canonique n'est pas de type fini~\cite[Proposition~1]{Kotwo}. Les singularit\'es slc se comportent donc moins bien vis-\`a-vis du programme des mod\`eles minimaux que leurs homologues normales que sont les singularit\'es klt, lc... C'est pour cette raison que nous avons d\^u traiter s\'epar\'ement, dans la preuve du th\'eor\`eme de r\'eduction stable, les vari\'et\'es normales au \S\ref{pargennormale} et les vari\'et\'es non normales au \S\ref{nonnormale}.

Le th\'eor\`eme de Hacon et Xu constitue un progr\`es sur ces questions dans le cas lc, dans un contexte adapt\'e \`a la preuve du th\'eor\`eme~\ref{redstable2} : ils travaillent dans une situation relative, et supposent connue l'existence du mod\`ele canonique de la fibre g\'en\'erique. \'Enon\c{c}ons le cas particulier\footnote{La finitude des alg\`ebres canoniques dans l'\'enonc\'e du th\'eor\`eme~\ref{thHX} est \'equivalente \`a l'existence des bons mod\`eles minimaux dans l'\'enonc\'e de~\cite[Theorem 1.1]{HX}, par~\cite[Lemma 2.9.1]{HMX}.}
de~\cite[Theorem 1.1]{HX} qui nous a \'et\'e utile.

\begin{theo}
\label{thHX}
Soit $f:\sX\to T$ un morphisme projectif, avec $\sX$ r\'egulier et $\Delta\subset \sX$ un diviseur \`a croisements normaux stricts. Supposons que $(\sX_\eta,\Delta_\eta)$ est de type g\'en\'eral. Si $A(\sX_\eta,\Delta_\eta)$ est de type fini sur $K$, alors $A(\sX,\Delta)$ est de type fini  sur $R$.
\end{theo}

La preuve utilise de mani\`ere cruciale les techniques de~\cite{BCHM}. Nous nous contentons d'en d\'ecrire la structure.

Apr\`es avoir peut-\^etre remplac\'e $(\sX,\Delta)$ par un mod\`ele birationnel (un mod\`ele minimal, construit en adaptant les techniques de~\cite{BCHM}), on souhaite d\'emontrer qu'il existe un entier $m>0$ tel que $\omega^{[m]}_{\sX}(m\Delta)$ est inversible et sans point base : ceci entra\^ine en effet la finitude de l'alg\`ebre canonique. Il est bien s\^ur n\'ecessaire de savoir d\'emontrer que la restriction $\omega^{[m]}_{\sX}(m\Delta)|_{\Delta}$ est elle-m\^eme sans point base, et un th\'eor\`eme d'extension d\^u \`a Fujino~\cite[Theorem 1.1]{Fubpf} montre que cela  serait en fait suffisant.

On voudrait obtenir cette information dans le cadre d'une r\'ecurrence sur la dimension. Malheureusement, $\Delta$ n'est en g\'en\'eral pas normal : il a seulement des singularit\'es slc et on ne peut lui appliquer l'hypoth\`ese de r\'ecurrence. L'id\'ee de Hacon et Xu est de plut\^ot appliquer l'hypoth\`ese de r\'ecurrence \`a la normalisation de $\Delta$, puis de redescendre l'information obtenue \`a $\Delta$ \`a l'aide de la technique de recollement de Koll\'ar que nous avons d\'ecrite aux \S\S\ref{techglue}--\ref{recoller}.

\section{Construction de l'espace de modules}
\label{consedm}

Cette section est consacr\'ee \`a la preuve du th\'eor\`eme~\ref{thedm}. Si la strat\'egie est connue depuis longtemps~\cite{KSB, Kocomplete, Viehweg}, beaucoup de d\'etails cruciaux n'ont \'et\'e mis au point que tr\`es r\'ecemment~\cite{KoHH, Kosing, HX, HMX, Fujino}. Le lecteur pourra consulter avec profit les textes de survol~\cite{YPG, Kosurvey}, ainsi que le livre~\cite{Kobook} pour une pr\'esentation d\'etaill\'ee.

La d\'emonstration exploite \`a nouveau les plongements pluricanoniques des vari\'et\'es stables. Ils permettent de param\'etrer les vari\'et\'es stables de fonction de Hilbert fix\'ee par une union de sous-sch\'emas localement ferm\'es d'un sch\'ema de Hilbert (\S\S\ref{borne}--\ref{repr}) qu'on quotiente ensuite par le groupe des transformations projectives pour construire l'espace de modules recherch\'e (\S\S\ref{champ}--\ref{grossier}).

\subsection{Caract\`ere born\'e}
\label{borne}

Fixons une fonction de Hilbert $F:\Z\to\Z$. La premi\`ere \'etape de la construction de $\overline{M}_F$ est la recherche d'une famille de vari\'et\'es projectives, dans laquelle toutes les vari\'et\'es stables de fonction de Hilbert $F$ apparaissent et qui soit \textit{born\'ee} au sens o\`u la base de la famille est elle-m\^eme une vari\'et\'e, donc de type fini sur le corps de base $k$.
 Pour cela, il suffit de d\'emontrer l'\'enonc\'e suivant.

\begin{theo}
\label{thborne}
Il existe un entier $m$ tel que pour toute vari\'et\'e stable $X$ de fonction de Hilbert $F$, le faisceau $\omega_X^{[m]}$ est inversible, tr\`es ample et sans cohomologie sup\'erieure.
\end{theo}

En effet, la famille universelle $g:\sY\to H$ au-dessus du sch\'ema de Hilbert $H$ param\'etrant les sous-sch\'emas ferm\'es de $\bP_{k}^{F(m)-1}$ de fonction de Hilbert $n\mapsto F(nm)$ a alors les propri\'et\'es requises. 

En restriction aux vari\'et\'es lisses, le th\'eor\`eme~\ref{thborne} est un cas particulier du grand th\'eor\`eme de Matsusaka~\cite{Matsusaka}. Le cas des courbes est facile (on peut prendre $m=3$) et c'est Alexeev qui a r\'esolu le cas des surfaces~\cite{Ale}. En g\'en\'eral, le th\'eor\`eme~\ref{thborne} a \'et\'e d\'emontr\'e par Hacon, McKernan et Xu~\cite{HMX}. 
On ne donnera ici aucune indication sur sa preuve, qui repose sur le programme des mod\`eles minimaux : on renvoie le lecteur au texte de survol~\cite{HMX2}.

\subsection{Repr\'esentabilit\'e de la stabilit\'e}
\label{repr}

La seconde \'etape de la preuve consiste \`a isoler, dans le sch\'ema de Hilbert $H$ construit au \S\ref{borne}, le lieu param\'etrant des vari\'et\'es stables $m$-canoniquement plong\'ees de fonction de Hilbert $F$.

\begin{theo}
\label{threpr}
Il existe un morphisme $\iota:H'\to H$ tel que le changement de base $g':\sY'\to H'$ de $f$ par $\iota$
soit une famille de vari\'et\'es stables $m$-canoniquement plong\'ees de fonction de Hilbert $F$, et qui soit universel pour cette propri\'et\'e.
\end{theo}

On proc\`ede par \'etapes, en effectuant plusieurs changements de base successifs, chacun am\'eliorant les propri\'et\'es du morphisme $g:\sY\to H$. Remarquons que $g$ est d\'ej\`a plat par d\'efinition du sch\'ema de Hilbert.

\subsubsection{}  
\label{deminormal}
On commence par remplacer $H$ par l'ouvert $H_1\subset H$ param\'etrant des vari\'et\'es r\'eduites, $S_2$ et \'equidimensionnelles~\cite[Th\'eor\`eme 12.2.1]{EGA43}, et dont les singularit\'es en codimension $1$ sont au plus des croisements normaux doubles (pour ce dernier point, on remarque que cette condition est \'equivalente \`a avoir des singularit\'es au plus nodales aux points de codimension $1$ et on utilise le fait que les d\'eformations des singularit\'es nodales sont au plus nodales ; voir~\cite[\S 1.41.2]{Kosing} pour un \'enonc\'e pr\'ecis).

Notant $g_1:\sY_1\to H_1$ le changement de base, on peut alors d\'efinir le faisceau canonique relatif $\omega_{\sY_1/H_1}$ et ses puissances r\'eflexives $\omega_{\sY_1/H_1}^{[n]}$ pour $n\in\Z$, comme au \S\ref{pardeffamst}.

\subsubsection{}
\label{ssHH}  
La seconde \'etape consiste \`a assurer que les faisceaux $\omega_{\sY_1/H_1}^{[n]}$ soient plats de formation commutant \`a tout changement de base, pour tout $1\leqslant n\leqslant m$. Cette \'etape est cruciale 
si l'on souhaite munir $\overline{M}_F$ d'une structure sch\'ematique raisonnable.
Elle a \'et\'e enti\`erement clarifi\'ee par Koll\'ar~\cite{KoHH} ; d'autres approches avaient \'et\'e propos\'ees par Hassett et Kov\'acs~\cite{HaKo}\footnote{L'article~\cite{HaKo} utilise la condition de Viehweg plut\^ot que celle de Koll\'ar (voir \S\ref{pardeffamst}).} et par Abramovich et Hassett~\cite{AbraHa}.

Koll\'ar construit, pour $1\leqslant n\leqslant m$, des d\'ecompositions $\Hull(\omega_{\sY_1/H_1}^{[n]})\to H_1$ de $H_1$ en sous-sch\'emas localement ferm\'es au-dessus desquels les faisceaux coh\'erents $\omega_{\sY_{1,s}}^{[n]}$ pour $s\in H_1$ s'organisent en une famille plate,
et qui sont universelles pour cette propri\'et\'e. Il ne reste plus qu'\`a d\'efinir $H_2:=\Hull(\omega_{\sY_1/H_1}^{[1]})\times_{H_1}\dots\times_{H_1}\Hull(\omega_{\sY_1/H_1}^{[m]})$ comme \'etant la d\'ecomposition de $H_1$ en sous-sch\'emas localement ferm\'es qui les raffine toutes, et \`a consid\'erer le morphisme $g_2:\sY_2\to H_2$ obtenu par changement de base.

Un des attraits du point de vue de Koll\'ar est sa g\'en\'eralit\'e : il n'utilise pas de propri\'et\'es particuli\`eres des vari\'et\'es stables, ni des faisceaux pluricanoniques. Il consid\`ere plut\^ot un morphisme projectif $p:X\to S$ arbitraire et un faisceau coh\'erent $\sF$ sur $X$. Dans cette situation, il construit une d\'ecomposition $\Hull(\sF)\to S$ de $S$ en sous-sch\'emas localement ferm\'es au-dessus de laquelle les hulls $\sF_s^{[**]}$ 
des $\sF_s$ pour $s\in S$ s'organisent en une famille plate, et qui est universelle pour cette propri\'et\'e~\cite[Theorem~21]{KoHH}. Si $X_s$ est $S_2$ et $\Supp(\sF_s)=X_s$, le hull $\sF_s^{[**]}$  n'est autre que le double dual $\sF_s^{**}$ de $\sF_s$ (voir~\cite[Definition 14]{KoHH} ou ci-dessous pour une d\'efinition g\'en\'erale des hulls). 
Ceci s'applique dans la situation que nous avons consid\'er\'ee. Avec les notations ci-dessus, on a donc
$(\omega_{\sY_1/H_1, s}^{[n]})^{[**]}=(\omega_{\sY_1/H_1, s}^{[n]})^{**}=\omega_{\sY_1,s}^{[n]}$ et la d\'ecomposition $\Hull(\omega_{\sY_1/H_1}^{[n]})$ a bien les propri\'et\'es voulues.
 
Pour construire $\Hull(\sF)$, Koll\'ar en identifie une compactification naturelle : l'espace de modules $\QHusk(\sF)$ des husks quotients coh\'erents de $\sF$. Un \textbf{husk quotient} coh\'erent de $\sF$ relativement \`a $S$ est un morphisme de faisceaux coh\'erents $q:\sF\to\sG$ sur $X$ o\`u $\sG$ est $f$-plat, tel que pour tout $s\in S$, le faisceau $\sG_s$ sur $X_s$ est pur
et $q$ est surjectif aux points g\'en\'eriques du support de $\sG_s$~\cite[Definition 9]{KoHH}. Ce husk quotient est un \textbf{hull} si pour tout $s\in S$, le morphisme $q_s$ est un isomorphisme aux points g\'en\'eriques du support de $\sG_s$, est surjectif aux points de codimension $1$ du support de $\sG_s$ et est maximal pour ces propri\'et\'es~\cite[Definition 17]{KoHH}. Par des techniques inspir\'ees de la construction des sch\'emas $\Quot$ de Grothendieck, on d\'emontre~\cite[Theorem 10]{KoHH} que le foncteur qui \`a un $S$-sch\'ema $T$ associe l'ensemble des hulls quotients coh\'erents $q:\sF_T\to \sG$ de $\sF_T$ relativement \`a $T$ est repr\'esentable par une union d\'enombrable d'espaces alg\'ebriques propres sur $S$, qu'on note $\QHusk(\sF)$ : c'est l'espace de modules $\QHusk(\sF)$ des husks quotients coh\'erents de $\sF$. On v\'erifie enfin que le
sous-foncteur des hulls est repr\'esentable par un ouvert $\Hull(\sF)$ de $\QHusk(\sF)$ et que $\Hull(\sF)$ est une d\'ecomposition de $S$ en sous-sch\'emas localement ferm\'es ~\cite[Theorem 21]{KoHH}.

\subsubsection{} 
\label{inversible}
 On remplace ensuite $H_2$ par l'ouvert $H_3\subset H_2$ au-dessus duquel $\omega_{\sY_2/H_2}^{[m]}$ est un faisceau inversible. Pour construire un tel ouvert, on proc\`ede comme suit.

Par le lemme de Nakayama, l'ensemble $\{y\in\sY_2| \dim_{\kappa(y)}(\omega_{\sY_2/H_2}^{[m]}\otimes \kappa(y))>1\}$ est ferm\'e (voir~\cite[III Example 12.7.2]{Hartshorne}). On note $H_3\subset H_2$ le compl\'ementaire de son image dans $H_2$ et $g_3:\sY_3\to H_3$ le changement de base. Montrons que $\sF:=\omega_{\sY_3/H_3}^{[m]}$ est inversible. Soit $y\in\sY_3$ et soit $s\in \sF_{y}$ engendrant $\sF_{y}\otimes \kappa(y)$. Par le lemme de Nakayama, le morphisme $\sO_{\sY_3,y}\xrightarrow{s}\sF_y$ est surjectif ; on note $\sN$ son noyau.
 La suite courte 
$$0\to \sN\otimes\kappa(g_3(y))\to\sO_{\sY_3,y}\otimes\kappa(g_3(y))\to\sF_y\otimes\kappa(g_3(y))\to 0$$
 est exacte car $\sF_y$ est $\sO_{H_3,g_3(y)}$-plat par \S\ref{ssHH}. Comme le support ensembliste de $\sF$ est $\sY_3$ et que $\sO_{\sY_3,y}\otimes\kappa(g_3(y))$ est r\'eduit, on d\'eduit que $\sN\otimes\kappa(g_3(y))=0$. A fortiori, $\sN\otimes\kappa(y)=0$, donc $\sN=0$ par le lemme de Nakayama. On a bien montr\'e que $\sF$ est inversible en $y$.

  On voit maintenant en \'ecrivant $n=am+b$ avec $1\leqslant b\leqslant m$ que $\omega^{[n]}_{\sY_3/H_3}$ est plat de formation commutant \`a tout changement de base, pour tout $n\in \Z$.

\subsubsection{} 
\label{ouvertslc}
  On se restreint alors \`a l'ouvert $H_4\subset H_3$ le long duquel les fibres de $g_3$ sont \`a singularit\'es slc (et on note $g_4:\sY_4\to H_4$ le changement de base). L'existence d'un tel ouvert est d\'emontr\'ee dans~\cite[Proposition A.1.1]{AbraHa}, o\`u la preuve est attribu\'ee \`a Alexeev.
 L'argument repose crucialement sur l'inversion de l'adjonction (voir \S~\ref{paradj}).

Voir que le lieu $\{x\in H_3|\sY_{3,x}\textrm{\hspace{.2em}a des singularit\'es slc}\}$ est un sous-ensemble constructible de $H_3$ est ais\'e. En effet, si $\eta$ est le point g\'en\'erique d'une composante irr\'eductible de $H_3$, une log r\'esolution de $\sY_{3,\eta}$ s'\'etend en une log r\'esolution des fibres de~$g_3$ au-dessus d'un voisinage $U$ de $\eta$ dans la vari\'et\'e r\'eduite $H_3^{\red}$. En calculant les discr\'epances des fibres de~$g_3$ au-dessus de $U$ sur ces log r\'esolutions, on montre que l'une est slc si et seulement si les autres le sont. On conclut par r\'ecurrence noeth\'erienne.

Il reste \`a d\'emontrer que ce lieu est stable par g\'en\'erisation. On se ram\`ene \`a une situation relative sur le spectre $T$ d'un anneau de valuation discr\`ete, de point ferm\'e $t$ et de point g\'en\'erique $\eta$. On dispose d'un morphisme propre et plat $f:\sX\to T$ dont les fibres satisfont les conditions (i)-(iv) de la d\'efinition~\ref{defslc}, tel que $\sX_t$ est slc et $\omega_{\sX/T}^{[m]}$ inversible. Le lemme~\ref{XXt}, dont la preuve reposait sur l'inversion de l'adjonction, assure que $(\sX,\sX_t)$ est slc, donc que $\sX_\eta$ est slc.

\subsubsection{}

On dispose de deux fibr\'es en droites naturels sur $\sY_4$ : le faisceau $m$-canonique $\omega_{\sY_4/H_4}^{[m]}$, qui est inversible par le \S\ref{inversible}, et le fibr\'e tautologique $\sO_{\sY_4/H_4}(1)$ induit par $\sO_{\bP^{F(m)-1}_k}(1)$.
On souhaite maintenant se restreindre au sous-sch\'ema localement ferm\'e $H_5\subset H_4$ au-dessus duquel ces deux fibr\'es en droites co\"incident, Zariski-localement sur $H_5$.
L'existence d'un tel sous-sch\'ema n'est pas \'evidente, car comme les fibres de $g_4$ peuvent ne pas \^etre irr\'eductibles, le foncteur de Picard $\Pic_{\sY_4/H_4}$ pourrait ne pas \^etre s\'epar\'e. Par cons\'equent, $H_5$ pourrait ne pas \^etre ferm\'e dans $H_4$. 

Par~\cite[Corollaire 7.8.7]{EGA32}, le faisceau $g_{4,*}\sO_{\sY_4}$ est localement libre de formation commutant au changement de base. Les fibres de $g_4$ \'etant r\'eduites par le \S\ref{deminormal}, le morphisme naturel $H'_4:=\Spec_{H_4}(g_{4,*}\sO_{\sY_4})\to H_4$ est fini \'etale. En consid\'erant la factorisation de Stein $g'_4:\sY_4\to H'_4$ de $g_4$, on se ram\`ene ais\'ement au cas o\`u les fibres de $g_4$ sont connexes, ce qu'on suppose d\'esormais. Dans ce cas, l'argument qui suit est donn\'e dans~\cite[Lemma 1.19]{Viehweg}.

On note $\sL:=\omega_{\sY_4/H_4}^{[m]}\otimes\sO_{\sY_4/H_4}(-1)$. Si $0\to\sL\to\sE^0\to\sE^1\to\dots$ est une r\'esolution de $\sL$ par des sommes de fibr\'es en droites assez amples et si l'on pose $\sF^i:=g_{4,*}\sE^i$, la cohomologie du complexe de fibr\'es vectoriels $0\to \sF^0\to \sF^1\to\dots$ sur $H_4$ calcule les $R^ig_{4,*}\sL$, et ce apr\`es tout changement de base (par cohomologie et changement de base). Soit $\sQ:=\Coker(\sF^0\to\sF^1)$. On commence par se restreindre \`a l'ouvert o\`u $\sQ$ a rang $\leqslant\rg(\sF^1)-\rg(\sF^0)+1$, puis au ferm\'e d\'efini par l'annulation des mineurs de taille $\rg(\sF^0)$ de $\sF^0\to\sF^1$. Le faisceau $\sQ$ est maintenant localement libre de rang $\rg(\sF^1)-\rg(\sF^0)+1$ par~\cite[Proposition 20.8]{Eisenbud}. Il suit que le noyau $\sK:=\Ker(\sF^0\to\sF^1)$ est localement libre de rang $1$, de formation commutant \`a tout changement de base.  
On d\'eduit que $g_{4,*}\sL=\sK$ est inversible et de formation commutant \`a tout changement de base.
 Il suffit pour conclure de se restreindre \`a l'ouvert $H_5$ au-dessus duquel le morphisme d'adjonction $g_4^*g_{4,*}\sL\to\sL$ est un isomorphisme.
 On note bien s\^ur $g_5:\sY_5\to H_5$ le changement de base.

\subsubsection{}
Consid\'erons l'ouvert $H_6\subset H_5$ au-dessus duquel $\sO_{\sY_5/H_5}(1)$ n'a pas de cohomologie sup\'erieure~\cite[Theorem 12.8]{Hartshorne}. Notant $g_6:\sY_6\to H_6$ le changement de base, le faisceau $g_{6,*}\sO_{\sY_6/H_6}(1)$ est un fibr\'e vectoriel par~\cite[Theorem 12.11]{Hartshorne}. On  se restreint finalement \`a l'ouvert $H'\subset H_6$ o\`u le morphisme $H^0(\bP_{k}^{F(m)-1}, \sO(1))\to g_{6,*}\sO_{\sY_6/H_6}(1)$ de fibr\'es vectoriels sur $H_6$ est un isomorphisme. Ce dernier point assure que la famille $g':\sY'\to H'$ obtenue par changement de base est $m$-canoniquement plong\'ee et ach\`eve la preuve du th\'eor\`eme~\ref{threpr}.

\subsection{Le champ de modules}
\label{champ}

On peut maintenant construire le champ de modules $\overline{\sM}_F$. Le groupe $\PGL_{F(m)}$ agit sur $\bP^{F(m)-1}_k$ par changement de coordonn\'ees, donc aussi sur son sch\'ema de Hilbert $H$.
Comme le morphisme $\iota:H'\to H$ est d\'efini par une propri\'et\'e universelle, cette action se rel\`eve naturellement en une action sur $H'$.
J'affirme que la cat\'egorie fibr\'ee en groupo\"ides $\overline{\sM}_F$ d\'efinie en~\eqref{fonctpoints} s'identifie canoniquement au champ quotient $[H'/\PGL_{F(m)}]$\footnote{Rappelons qu'un B-point du champ quotient $[H'/\PGL_{F(m)}]$ est la donn\'ee d'un $\PGL_{F(m)}$-torseur $I$ sur $B$, et d'un morphisme $\PGL_{F(m)}$-\'equivariant $\mu:I\to H'$.} et est donc un champ alg\'ebrique~\cite[(4.6.1)]{LMB}.

 Intuitivement, cette assertion est claire : $H'$ param\`etre les vari\'et\'es stables de fonction de Hilbert $F$ qui sont $m$-canoniquement plong\'ees dans $\bP^{F(m)-1}_k$, et deux telles sous-vari\'et\'es de $\bP^{F(m)-1}_k$ sont isomorphes comme vari\'et\'es abstraites si et seulement si elles diff\`erent par un changement de coordonn\'ees projectives.

 Justifions-le plus formellement. On se contente ici de construire un $1$-morphisme $\overline{\sM}_F\to [H'/\PGL_{F(m)}]$ ; il est ais\'e de v\'erifier que c'est un isomorphisme. Consid\'erons un $B$-point de $\overline{\sM}_F$, qui correspond  \`a une famille stable $f:\sX\to B$ dont les fibres ont fonction de Hilbert $F$ comme en~\eqref{fonctpoints}. Le faisceau $\omega_{\sX/B}^{[m]}$ est plat et inversible sur les fibres de $f$ par le th\'eor\`eme~\ref{thborne}. Il est donc inversible par l'argument du \S\ref{inversible}.
Par cohomologie et changement de base~\cite[III Theorem 12.11]{Hartshorne}, 
qui s'applique par le th\'eor\`eme~\ref{thborne} et comme $\omega_{\sX/B}^{[m]}$ est $f$-plat, le faisceau $f_*\omega_{\sX/B}^{[m]}$ est localement libre de rang $F(m)$. Le faisceau inversible $\omega_{\sX/B}^{[m]}$ est tr\`es ample relativement \`a $f$, et plonge $\sX$ dans le fibr\'e projectif $\bP_B(f_*\omega_{\sX/B}^{[m]})$ sur $B$. Soit $I$ le sch\'ema des isomorphismes $\Isom_B (\bP^{F(m)-1}_B,\bP_B(f_*\omega_{\sX/B}^{[m]}))$ : c'est un $\PGL_{F(m)}$-torseur sur $B$. En tirant $f:\sX\to B$ sur $I$ et en utilisant l'isomorphisme universel au-dessus de $I$, on obtient un morphisme $\lambda:I\to H$, qui factorise \`a travers un morphisme $\mu:I\to H'$ par la propri\'et\'e universelle de $H'$. Le couple $(I,\mu)$ est le $B$-point de $[H'/\PGL_{F(m)}]$ que nous cherchions \`a construire.

Remarquons que l'action de $\PGL_{F(m)}$ se rel\`eve naturellement \`a l'espace total $\sY'$ de la famille $g':\sY'\to H'$. Notant $\overline{\sU}_F:=[\sY'/\PGL_{F(m)}]$, on voit que le champ $\overline{\sM}_F$ porte une famille stable universelle
$u:\overline{\sU}_F\to \overline{\sM}_F$.

\medskip

Le champ $\overline{\sM}_F$ est  propre. En effet, le crit\`ere valuatif de propret\'e~\cite[Theorem 11.5.1]{Olsson} est v\'erifi\'e par le th\'eor\`eme de r\'eduction stable, sous sa forme donn\'ee au th\'eor\`eme~\ref{redstable2} (la s\'eparation de $\overline{\sM}_F$ correspondant \`a la propri\'et\'e d'unicit\'e dans le th\'eor\` eme de r\'eduction stable).

On en d\'eduit la finitude des groupes d'automorphismes des vari\'et\'es stables. En effet, le crit\`ere valuatif de propret\'e pour ces groupes d'automorphismes est un cas particulier du crit\`ere valuatif de s\'eparation pour $\overline{\sM}_F$ et est donc v\'erifi\'e. Ces groupes d'automorphismes sont donc propres. Comme ils sont de plus affines (ce sont des sous-groupes de $\PGL_{F(m)}$), ils sont finis. 

Le sch\'ema en groupes des automorphismes d'une vari\'et\'e stable est de plus r\'eduit, comme tout sch\'ema en groupes en caract\'eristique nulle. 
Appliquant~\cite[Theorem~8.3.3, Remark 8.3.4]{Olsson}, on voit que $\overline{\sM}_F$ est un champ de Deligne-Mumford.

\subsection{L'espace de modules grossier}
\label{grossier}

C'est un th\'eor\`eme de Keel et Mori~\cite{KeMo} que tout champ s\'epar\'e de type fini sur un corps
admet un espace de modules grossier, qui est un espace alg\'ebrique s\'epar\'e de type fini sur ce corps. Notons $\overline{M}_F$ l'espace de modules grossier de $\overline{\sM}_F$. La propret\'e de $\overline{\sM}_F$ entra\^ine celle de $\overline{M}_F$.

Il reste \`a d\'emontrer la projectivit\'e de $\overline{M}_F$. Les premi\`eres approches \`a ce probl\`eme ont repos\'e sur une autre strat\'egie de construction de $\overline{M}_F$ que celle pr\'esent\'ee ici : la th\'eorie g\'eom\'etrique des invariants de Mumford~\cite{GIT}, dont l'avantage est de produire
des vari\'et\'es qui sont, par construction, munies d'un fibr\'e ample. C'est ainsi que  Knudsen~\cite{Knudsen} et Gieseker et Mumford~\cite[\S 5]{Mumford} ont d\'emontr\'e la projectivit\'e de $\overline{M}_g$. Cette technique a \'egalement permis \`a Gieseker de d\'emontrer la quasi-projectivit\'e des espaces de modules de surfaces lisses canoniquement polaris\'ees~\cite{Gieseker}, et Viehweg a r\'eussi, dans un v\'eritable tour de force, \`a faire de m\^eme en dimension arbitraire~\cite{Viehweg}. Cette m\'ethode se heurte cependant \`a des difficult\'es s\'erieuses, d\'ej\`a soulev\'ees par Mumford~\cite[\S 3]{Mumford} et Shah~\cite{Shah}, dans le cas des vari\'et\'es singuli\`eres.

\medskip

Koll\'ar a propos\'e dans~\cite{Kocomplete} une autre m\'ethode pour d\'emontrer la projectivit\'e d'un espace de modules qu'on sait \^etre propre. Expliquons comment elle s'applique \`a $\overline{M}_F$.

L'espace alg\'ebrique $\overline{M}_F$ porte de nombreux fibr\'es en droites naturels. Rappelons que $u:\overline{\sU}_F\to\overline{\sM}_F$ est la famille universelle construite au \S\ref{champ}. Si $n>0$ est un entier assez divisible, le faisceau $\omega^{[n]}_{\overline{\sU}_F/\overline{\sM}_F}$ est un fibr\'e en droites $u$-ample dont les restrictions aux fibres de $u$ n'ont pas de cohomologie sup\'erieure. Pour un tel $n$, le faisceau $u_*\omega^{[n]}_{\overline{\sU}_F/\overline{\sM}_F}$ est un fibr\'e vectoriel sur $\overline{\sM}_F$.
Pour $N>0$ assez divisible, le fibr\'e en droites $\det(u_*\omega^{[n]}_{\overline{\sU}_F/\overline{\sM}_F})^{\otimes N}$ sur $\overline{\sM}_F$ descend en un unique fibr\'e en droites $\sL_{n, N}$ sur $\overline{M}_F$
par~\cite{Rydh} (voir aussi~\cite[Lemma 9.26]{Viehweg}).

\begin{theo}
\label{amplitude}
Si $n$ est assez divisible, $\sL_{n, N}$ est un fibr\'e ample sur $\overline{M}_F$.
\end{theo}

  Ce th\'eor\` eme a \'et\'e d\'emontr\'e par Koll\'ar dans~\cite[Corollary 5.6]{Kocomplete} pour les espaces de modules de surfaces stables et par Fujino~\cite{Fujino} en g\'en\'eral. 

Koll\'ar a remarqu\'e qu'il suffisait de v\'erifier que pour toute courbe projective lisse $C$ et pour tout morphisme $\psi:C\to\overline{\sM}_F$, le fibr\'e vectoriel $\psi^*(u_*\omega^{[n]}_{\overline{\sU}_F/\overline{\sM}_F})$ est nef~\cite[Theorem 2.6]{Kocomplete}. Qu'on puisse d\'eduire le th\'eor\`eme~\ref{amplitude}, qui est un \'enonc\'e de positivit\'e \textit{stricte}, d'un tel r\'esultat de positivit\'e \textit{au sens large} est remarquable. L'argument repose sur le crit\`ere d'amplitude de Nakai-Moishezon. Le gain de positivit\'e est fourni par la variation, dans une famille de vari\'et\'es stables, des \'equations de ces vari\'et\'es dans leurs plongements pluricanoniques (voir~\cite[\S 2.9]{Kocomplete}).

La v\'erification du fait que $\psi^*(u_*\omega^{[n]}_{\overline{\sU}_F/\overline{\sM}_F})$ est nef est due \`a Koll\'ar~\cite[Theorem~4.12]{Kocomplete} pour les familles de surfaces et \`a Fujino~\cite[Theorem 1.7]{Fujino} en g\'en\'eral. Elle s'appuie sur des th\'eor\`emes de semipositivit\'e en th\'eorie de Hodge, qui remontent \`a Fujita~\cite{Fujita} et qui sont d\'emontr\'es dans la g\'en\'eralit\'e requise dans~\cite[Corollary~5.23]{FF}.

\medskip

Patakfalvi et Xu~\cite{PatakXu} ont montr\'e l'amplitude d'un autre fibr\'e en droites, d\'efini seulement sur la normalisation de $\overline{M}_F$ : le \textit{fibr\'e en droites CM}. Comme $\overline{M}_F$ est propre, cela fournit une autre preuve de sa projectivit\'e (voir~\cite[Proposition 2.6.2]{EGA31}).

\bigskip

{\bf Remerciements.}
Merci \`a Olivier Debarre, St\'ephane Druel, Javier Fres\'an, Christopher Hacon, J\'anos Koll\'ar et Bertrand R\'emy pour leurs utiles commentaires.

\end{document}